\documentclass[11pt]{article}

\usepackage[margin=1in]{geometry}

\usepackage{graphicx}
\usepackage{amssymb,amsmath}
\usepackage{amsthm,enumerate}
\usepackage{hyperref,xcolor}
\usepackage{mathrsfs}
\usepackage{extarrows}
\usepackage{textcomp}
\allowdisplaybreaks[4]

\numberwithin{equation}{section}

\newtheorem{theorem}{Theorem}[section]
\newtheorem{lemma}[theorem]{Lemma}

\newtheorem{definition}[theorem]{Definition}

\newtheorem{corollary}[theorem]{Corollary}
\newtheorem{remark}[theorem]{Remark}

\def\<{\langle}
\def\>{\rangle}

\def\XXint#1#2#3{{\setbox0=\hbox{$#1{#2#3}{\int}$ }
\vcenter{\hbox{$#2#3$ }}\kern-.6\wd0}}

\begin{document}

\title{Asymptotic behavior at infinity and existence of solutions to\\ the Lagrangian mean curvature flow in $\mathbb R^{n+1}_-$\footnote{Supported in part by the National Natural Science Foundation of China (No. 12501253 and 12371200) and by the Beijing Natural Science Foundation (No. 1244038 and No. 1254049).}}

\author{Jiguang Bao, Zixiao Liu}
\date{\today}

\maketitle

\begin{abstract}
    This paper investigates the asymptotic behavior at infinity of ancient solutions to the Lagrangian mean curvature flow. 
    Under conditions that admit Liouville type rigidity theorems, we prove that every classical solution converges at infinity
    to the sum of a quadratic polynomial in $x$ and a linear function in $t$, with an explicitly derived exponential rate of convergence.
    As a critical part of the proof framework of this paper,  we establish the existence of a global viscosity solution with prescribed asymptotic behavior at infinity, featuring two key innovations: (i) applicability to all dimensions $n\geq 2$, and (ii) no requirement that the Hessian matrix of the prescribed quadratic term be positive definite or close to a scalar multiple of the identity matrix.  
    These results establish the relationship between Liouville type rigidity, asymptotic analysis at infinity, and the existence of viscosity solutions. 


\end{abstract}

\section{Introduction}

We focus on the Lagrangian mean curvature flow equation
\begin{equation}\label{equ-LagFlow-Extended}
    u_t=\sum_{i=1}^n\arctan\lambda_i(D^2u)+f(x,t)\quad\text{in }\mathbb R^{n+1}_-,
\end{equation}
where $u_t$ denotes the partial derivative of $u(x,t)$ with respect to  time variable $t$, $D^2u$ is the Hessian matrix of $u(x,t)$ with respect to space variable $x$,  $\lambda_1(D^2 u),\cdots,\lambda_n(D^2u)$ denotes the eigenvalues of $D^2u$, $f(x,t)$ is a known function, and $\mathbb R^{n+1}:=\mathbb R^n\times(-\infty,0]$. The equation origins naturally from the evolution of a gradient graph $(x,Du(x,t))\in\mathbb R^n\times\mathbb R^n$ by the mean curvature flow, and the gradient graph remains a Lagrangian submanifold in $\mathbb R^n\times\mathbb R^n$ with the standard symplectic structure for all $t\in (-\infty,0]$. (See for instance \cite{Harvey-Lawson-CalibratedGeometries,Smoczyk-LagMeanCurvFlow-Existence,Smoczyk-Wang-MeanCurFlow-ConvePotential,Chen-Pang-UniqueLagranMeanFlow}, etc.) 
For a smooth stationary solution to \eqref{equ-LagFlow-Extended}  i.e.,   the Lagrangian mean curvature equation
\begin{equation}\label{equ-Spl-Lagmean}
\sum_{i=1}^n\arctan\lambda_i(D^2u)=f(x)\quad\text{in }\mathbb R^n,
\end{equation}
the graph becomes a Lagrangian submanifold with mean curvature $(0,Df(x))^{\perp}$ in $\mathbb R^{n}\times\mathbb R^n$. Especially, when $f(x)\equiv \Theta$ is a constant, the gradient graph $(x,Du(x))$ is minimal in $\mathbb R^n\times\mathbb R^n$, and  equation \eqref{equ-Spl-Lagmean} is known as the special Lagrangian equation.  (See for instance \cite{Warren-Calibrations-MA,Wang-Huang-Bao-SecondBoundary-LagCurvFlow}, etc.)

For the special Lagrangian equation, Yuan \cite{Yuan-GlobalSolution-SPL,Yuan-Bernstein-SPL} proved that any classical solution to  equation \eqref{equ-Spl-Lagmean} with $f\equiv \Theta$ satisfying 
\[
\Theta>\frac{(n-2)\pi}{2}\quad\text{or}\quad \lambda_i(x)\geq \left\{
\begin{array}{llll}
    -K, & \text{when }n\leq 3,\\
    -\epsilon(n), & \text{when }n\geq 4,
\end{array} 
\right.\quad\forall~1\leq i\leq n,
\]
must be a quadratic polynomial, where $K$ is an arbitrarily constant, $\epsilon(n)$ is a dimensional constant, and  $\lambda_i(x)$ is short for $\lambda_i(D^2u(x))$.  
Warren and Yuan \cite{Warren-Yuan-Liouville-SPL} proposed an alternative constraint condition, such that the Liouville type rigidity holds. Recently, it was generalized by Ding \cite{Ding-Liouville-SPL-HessianEst} into 
\[
3+\lambda_i^2(x)+2\lambda_i(x)\lambda_j(x)\geq 0,\quad\forall~1\leq i,j\leq n.
\]
We refer to   \cite{Li-Li-Yuan-BernsteinThm,Li-Dirichlet-SPL,Bao-Liu-Wang-DiricLagrangianEntire,Han-Marchenko-SolutionSPL-NearInfinity} and the references therein, for the investigation on asymptotic behavior at infinity and the Dirichlet type problems of equation \eqref{equ-Spl-Lagmean}. 

For the Lagrangian mean curvature flow in dimensions $n\geq 2$, Nguyen and Yuan \cite{Nguyen-Yuan-PrioriEst-LagranMeanCurvFlow} considered classical solution to \eqref{equ-LagFlow-Extended} with $f\equiv 0$ that satisfies
    \begin{equation}\label{equ-cond-HessianBdd}
    |D^2u(x,t)|\leq K<\infty,\quad\forall~(x,t)\in\mathbb R^{n+1}_-,
    \end{equation}
    and either of the following conditions when $n\geq 4$:
    \begin{enumerate}[(i)]
        \item\label{condition-1} $\displaystyle\sum_{i=1}^n\arctan\lambda_i(x,t)\geq\frac{(n-2)\pi}{2}$;
        \item\label{condition-2} $\lambda_i(x,t)\geq-\epsilon(n)$, for $1\leq i\leq n$;
        \item\label{condition-3} $3+\lambda_i^2(x,t)+2\lambda_i(x,t)\lambda_j(x,t)\geq 0$, for $1\leq i,j\leq n$,
    \end{enumerate}
    where $\lambda_i(x,t)$ is short for $\lambda_i(D^2u(x,t))$ and $\epsilon(n)$ is a dimensional constant from \cite{Yuan-Bernstein-SPL}.
Under the assumptions above, $u$ must be of form
\begin{equation}\label{equ-ParabolicForm}
u(x,t)=\tau t+\frac{1}{2}x'Ax+b\cdot x+c,\quad\forall~(x,t)\in\mathbb R^{n+1}_-,
\end{equation}
where  $A\in\mathrm{Sym}(n)$ is a symmetric matrix, $b\in\mathbb R^n, c\in\mathbb R$, $\tau=\sum_{i=1}^n\arctan\lambda_i(A)$, and $x'$ denote the transpose of $x$. Most recently, 
 Bhattacharya, Warren and Weser \cite{Bhattacharya-Warren-Weser-Liouville-LagranFlow-CPDE} also obtained the Liouville type rigidity for all $n\geq 2$, under the assumptions that $u$ being a classical solution to \eqref{equ-LagFlow-Extended} with $f\equiv 0$ that is convex with respect to $x$ and 
  \begin{equation}\label{equ-cond-Growth-BhattaWarrenWeser}
    \limsup_{t\rightarrow-\infty}\sup_{x\in\mathbb R^n}\dfrac{|u(x,t)|}{|x|^2+R_0}<\dfrac{1}{(6\sqrt n+2)^2},
  \end{equation}
  for some $R_0>0$.

While solutions to \eqref{equ-LagFlow-Extended} with $f\equiv 0$ in $\mathbb R^{n+1}_-$ have been thoroughly investigated, the asymptotic behavior at infinity for the case where $f$ has compact support remains largely open. This is very different from the research progress in the study of special Lagrangian equations \cite{Li-Li-Yuan-BernsteinThm} and the parabolic Monge-Amp\`ere equations \cite{Xiong-Bao-JCP-ParaboMA,An-Bao-Liu-EntireSol-ParboMA-UnbouGrowth}, highlighting a distinct research gap in the present context.

In this  paper, we reveal the asymptotic behavior of solutions to \eqref{equ-LagFlow-Extended}.
Throughout the text, we always suppose that $f(x,t)$ satisfies the condition $(F)$:
\[
f\in C^0(\mathbb R^{n+1}_-)\quad\text{and}\quad\mathrm{supp}(f)\subset B_S\times[-T,0],
\]
where $B_S\subset\mathbb R^n$ denote the ball centered at the origin with radius $S$, and 
\begin{equation}\label{equ-def-T}
    T:=\sup\{-t~:~f(x,t)\neq 0\text{ for some }x\in\mathbb R^n\}.
\end{equation}

\begin{theorem}\label{thm-Main1-AsymBehav}
    Let $n\geq 2$, $f$ satisfy the condition $(F)$, and $u\in C^{2,1}(\mathbb R^{n+1}_-)$ be a classical solution to \eqref{equ-LagFlow-Extended}. 
    Assume further that either of the following conditions holds,
    \begin{enumerate}[(a)]
        \item\label{cond-MainThm-1} $u$ satisfying \eqref{equ-cond-HessianBdd}, and either of conditions \eqref{condition-1}-\eqref{condition-3};
        \item\label{cond-MainThm-2} $u$ being convex with respect to $x$, and satisfying condition \eqref{equ-cond-Growth-BhattaWarrenWeser}.
    \end{enumerate}
    Then there exist $A\in\mathrm{Sym}(n)$, $b\in\mathbb R^n, c\in\mathbb R$, and $\tau:=\sum_{i=1}^n\arctan\lambda_i(A)$ such that
    \[
    E(x,t):=u(x,t)-\left(\tau t+\frac{1}{2}x'Ax+b\cdot x+c\right)
    \]
    satisfies $E(x,t)=0$ in $\mathbb R^n\times(-\infty,-T]$, and for some positive constants $C,\widetilde S>S$, 
    \begin{equation}\label{equ-Result-AsymBehav-u}
    |E(x,t)| 
           \leq  \frac{C}{(t+T)^{\frac{n}{2}-2}x'(I+A^2)x}
           \exp\left(-\frac{x'(I+A^2)x-\widetilde S|(I+A^2)x|}{4(t+T)}\right), 
    \end{equation}
    for sufficiently large $|x|$ and $t\in (-T,0]$.
\end{theorem}

\begin{remark}\label{Rem-1.2}
    Notably, for classical solutions in case \eqref{cond-MainThm-1}, either of  conditions \eqref{condition-1}-\eqref{condition-3} can be replaced by the holding of Liouville type rigidity. 
    Furthermore, estimate \eqref{equ-Result-AsymBehav-u} provides not only the asymptotic behavior at infinity, but also the vanishing rate as $t\rightarrow -T$ with $|x|$ sufficiently large. 
\end{remark}
\begin{remark}\label{Rem-1.3}
    By linearizing equation \eqref{equ-LagFlow-Extended}, we obtain 
    \[
    -E_t(x,t)+\sum_{i,j=1}^na_{ij}(x,t)D_{ij}E(x,t)=-f(x,t),\quad\text{in }\mathbb R^n\times [-T,0],
    \]
    with $[a_{ij}(x,t)]_{n\times n}\rightarrow (I+A^2)^{-1}$ as $-t+|x|\rightarrow\infty.$
    The main term in \eqref{equ-Result-AsymBehav-u} comes from the fundamental solution to the linearized equation as above, given by $-\frac{1}{(t+T)^{\frac{n}{2}}}\exp\left(-\frac{x'(I+A^2)x}{4(t+T)}\right)$, and it is optimal in the following sense. 
    
    For any $0<\alpha<\frac{1}{2}$, suppose $f_\alpha\leq 0$ satisfies the condition $(F)$, and
    \[
    f_\alpha(x,t)\leq -(t+T)^{\alpha},\quad\text{in }B_{\frac{S}{2}}\times[-T,0].
    \]
    Let $u, A, b, c, \tau$, and $E$ be as in Theorem \ref{thm-Main1-AsymBehav}. Especially, we assume that $u$ is strictly convex with respect to $x$. Then for any $0<\beta<1$, there exist positive constants $C$ and $\underline S>S$ such that $E$ satisfies \eqref{equ-Result-AsymBehav-u} and
    \[
    |E(x,t)|\geq \dfrac{C}{(t+T)^{\frac{n}{2}-\alpha-2}x'(I+A^2)x}\exp\left(
        -\frac{x'(I+A^2)x+\underline S|(I+A^2)x|}{4(1-\beta)(t+T)}
    \right),
    \]
    for sufficiently large $|x|$ and $t\in (-T,0]$. Since $\alpha$ and $\beta$ can be arbitrarily close to $0$, this proves the optimality of the main term in \eqref{equ-Result-AsymBehav-u}. The proof of this remark is given at the end of this paper.
\end{remark}

In general, viscosity solutions to fully nonlinear parabolic equations may not be classical solutions, a fundamental distinction that amplifies the difficulty of establishing Liouville type rigidity  (see for instance \cite{Gutierrez-Huang-JCP-ParaboMA,An-Bao-Liu-EntireSol-ParboMA-UnbouGrowth} for illustrative cases in parabolic Monge-Amp\`ere equations). 
For definitions of viscosity solutions to \eqref{equ-LagFlow-Extended}, we refer the reader to   \cite{User'sGuide-ViscositySol,Zhou-Gong-Bao-AncientSolution-paraboicMA}.

To date, no Liouville type rigidity result has been established for viscosity solutions to \eqref{equ-LagFlow-Extended} with $f\equiv 0$. 
Nevertheless, to facilitate broader theoretical applications, we derive the following asymptotic behavior result under the hypothesis that Liouville type rigidity holds for viscosity solutions under suitable  conditions. 
\begin{theorem}\label{thm-main5-GeneralThm}
    Let $n\geq 2$, $f\in C^0(\mathbb R^{n+1}_-)$ have compact support, and let $u\in C^0(\mathbb R^{n+1}_-)$ be a viscosity solution to \eqref{equ-LagFlow-Extended}. 
    Assume that the Liouville type rigidity holds for \eqref{equ-LagFlow-Extended} with $f\equiv 0$. Then for any $\zeta>0$,  there exists $C>0$   such that 
    \begin{equation}\label{equ-AsymBehav-Rough-C0Est-E}
        |E(x,t)| 
           \leq C|x|^{-\zeta},\quad\forall~(x,t)\in\mathbb R^n\times(-T,0], 
    \end{equation}
    where  $E$ is defined as in Theorem \ref{thm-Main1-AsymBehav}, $A\in\mathrm{Sym}(n)$, $b\in\mathbb R^n, c\in\mathbb R$, and $\tau:=\sum_{i=1}^n\arctan\lambda_i(A)$. 
\end{theorem}

As the second part of this paper, which is critical to the proof framework of the paper, we provide the existence of viscosity solution to equation \eqref{equ-LagFlow-Extended} with prescribed asymptotic behavior at infinity. 
Hereinafter, we denote 
\[
\mathcal R(x,t):=(-t+|x|^2)^{\frac{1}{2}},\quad\forall~(x,t)\in\mathbb R^{n+1}_-,
\]
which will be used to characterize the asymptotic behavior as $-t+|x|^2\rightarrow\infty$.  
\begin{theorem}\label{thm-Main3-EntireDiri}
    Let $n\geq 2$ and $f\in C^0(\mathbb R^{n+1}_-)$ satisfy 
    \begin{equation}\label{equ-cond-RHS-ConverSpeed}
        f(x,t)=O(\mathcal R^{-\beta}(x,t)),\quad\text{as }\mathcal R(x,t)\rightarrow\infty,
    \end{equation}
    for some $\beta>2$. For any $A\in\mathrm{Sym}(n), b\in\mathbb R^n, c\in\mathbb R$, and $\tau:=\sum_{i=1}^n\arctan\lambda_i(A)$, there exists a unique   viscosity solution $u\in C^0(\mathbb R^{n+1}_-)$ to 
    \begin{equation}\label{equ-Prob-EntireDiri-InThm}
        \left\{
            \begin{array}{llll}
                \displaystyle u_t=\sum_{i=1}^n\arctan\lambda_i(D^2u)+f(x,t), & \text{in }\mathbb R^{n+1}_-,\\
                \displaystyle 
                u(x,t)=\tau t+\frac{1}{2}x'Ax+b\cdot x+c+
                O(\mathcal R^{2-\beta}(x,t)), &   \text{as }\mathcal R(x,t)\rightarrow\infty.
            \end{array}
        \right.
    \end{equation}
    Especially, when $f\in C^0(\mathbb R^{n+1}_-)$ has compact support, the asymptotic behavior condition in \eqref{equ-Prob-EntireDiri-InThm} can be replaced into 
    \begin{equation}\label{equ-enhanced-AsymBehav}
    u(x,t)=\tau t+\frac{1}{2}x'Ax+b\cdot x+c+
                O(\mathcal R^{-\zeta}(x,t)),\quad\text{as }\mathcal R(x,t)\rightarrow\infty,
    \end{equation}
    for any $\zeta>0$.
\end{theorem} 

Importantly, unlike the Lagrangian mean curvature equation \eqref{equ-Spl-Lagmean}, our results impose no restrictions on the matrix $A$, neither requiring $A$ to be positive definite nor close to a constant multiple of the identity matrix $I$. Such constraints, standard in fully nonlinear elliptic equations (see for instance \cite{Li-Dirichlet-SPL,Bao-Liu-Wang-DiricLagrangianEntire,Li-Wang-EntireDirichlet}), are herein relaxed to accommodate broader scenarios. 

  As a further application of Theorem \ref{thm-Main3-EntireDiri}, we have the following asymptotic behavior of solutions to the initial value problems, but only work for  initial value $u(x,0)$ being a quadratic polynomial. 
\begin{corollary}\label{Coro-InitialValProb-AsymBehav}
    Let $\overline T>0$,  $f\in C^0(\mathbb R^n\times[0,\overline T])$ satisfy $f(x,0)\equiv 0$ in $\mathbb R^n$ and
    \begin{equation}\label{equ-cond-InitialValProb-Coro}
        |f(x,t)|\leq C|x|^{-\beta},\quad\forall~|x|\geq 1,\quad 0\leq t\leq\overline T,
    \end{equation}
    for some $\beta>2$ and $C>0$.
    Let   $u\in C^0(\mathbb R^n\times[0,\overline T])$ be a viscosity solution to 
    \[
    \left\{
        \begin{array}{llll}
            \displaystyle u_t=\sum_{i=1}^n\arctan\lambda_i(D^2u)+f(x,t), & \text{in }\mathbb R^n\times[0,\overline T],\\
            \displaystyle u(x,0)=\frac{1}{2}x'Ax+b\cdot x+c, & \text{in }\mathbb R^n,
        \end{array}
    \right.
    \]
    where $A\in\mathrm{Sym}(n), b\in\mathbb R^n$ and $c\in\mathbb R$.
    Then  there exists $C>0$ such that 
    \[
    \left|u(x,t)-\left(\tau t+\frac{1}{2}x'Ax+b\cdot x+c\right)\right|
    \leq C|x|^{2-\beta},
    \]
    for all $(x,t)\in\mathbb R^n\times [0,\overline T]$, where $\tau:=\sum_{i=1}^n\arctan\lambda_i(A)$.
\end{corollary}

We note that substantial research has explored solutions to equation  \eqref{equ-LagFlow-Extended}  of various structures, such as the 
 self-similar solutions, translating solutions, and rotating solutions. See for instance 
Chau, Chen and He \cite{Chau-Chen-He-LagrangianMeanCurFlow-LipschitzGraph},
Chau, Chen and Yuan \cite{Chau-Chen-Yuan-Rigidity-SelfShrinking-CurvFlow,Chau-Chen-Yuan-LagCurvFlow-EntireLip-II}, Yuan \cite{Yuan-SPLEqu-InCollection}, Bhattacharya and Shankar \cite{Bhattacharya-Shankar-OptimalRegularity-LagMeanCur},  Bhattacharya-Wall \cite{Bhattacharya-Wall-HessianEst-LagranFlow}, Tsai, Tsui and Wang \cite{Tsai-Tsui-Wang-EntireSol-twoconvex-LagCurvFlow}, Lotay, Schulze and Se\'ekelyhidi \cite{Lotay-Schulze-Szekelyhidi-AncientSol-LagCurvFlow} and the references therein.

The paper is organized as follows. In Section \ref{seclabel-InitialValProb}, we utilize the barrier function constructed in \cite{Barles-Biton-Ley-UniqueWithoutGrowth} and prove a key comparison principle for the initial value problem. Section \ref{seclabel-Sec-ExistenceSol} is dedicated to establish the existence of ancient solutions prescribing asymptotic behavior at infinity, as formulated in \eqref{equ-Prob-EntireDiri-InThm}, and prove Theorem \ref{thm-Main3-EntireDiri}. In Section \ref{seclabel-Theorem-AsymBehav}, we conduct an asymptotic behavior analysis at infinity by employing the comparison principle proved in Section \ref{seclabel-InitialValProb} and the viscosity solution we construct in Theorem \ref{thm-Main3-EntireDiri}. Finally, by establishing an iteration scheme, we refine the convergence rate to an exponential type, which culminating in the proof of Theorems \ref{thm-Main1-AsymBehav}
 and \ref{thm-main5-GeneralThm}. 

\section{Initial  value problem without growth condition at infinity}\label{seclabel-InitialValProb}

In this section, we establish a key comparison principle for the initial value problem of \eqref{equ-LagFlow-Extended}, which serves as a fundamental ingredient in the proof of Theorem \ref{thm-Main1-AsymBehav}. Notably, the principle dispenses with the usual growth condition at infinity, which is nonstandard for general parabolic equations, for instance, the heat equation. Such uniqueness and comparison principle results without growth condition  were first established by Barles, Biton and Ley \cite[Theorems 2.1 and 3.1]{Barles-Biton-Ley-UniqueWithoutGrowth} for a class of fully nonlinear parabolic equations. Chen and Pang \cite{Chen-Pang-UniqueLagranMeanFlow} further extended these findings to the initial value problem of Lagrangian mean curvature flow \eqref{equ-LagFlow-Extended} on $\mathbb R^n\times[0,T]$ with $f\equiv 0$. Building upon the barrier function developed in \cite{Barles-Biton-Ley-UniqueWithoutGrowth,Chen-Pang-UniqueLagranMeanFlow}, we state the following theorem, which can be seen as a special instance of Theorem 2.1 in \cite{Barles-Biton-Ley-UniqueWithoutGrowth}.

\begin{theorem}\label{thm-Main2-Uniqueness-Extended}
    Let $T>0$, $f\in C^0(\mathbb R^{n}\times[0,T])$, and let $u,v\in C^0(\mathbb R^n\times[0,T])$ respectively be viscosity subsolution and supersolution to 
    \begin{equation}\label{equ-InitialValProb-Extended}
    \left\{
        \begin{array}{lll}
    \displaystyle    u_t=\sum_{i=1}^n\arctan\lambda_i(D^2u)+f(x,t), & \text{in }\mathbb R^n\times(0,T],\\
    u(x,0)=u_0(x), & \text{in }\mathbb R^n.
        \end{array}
    \right.
    \end{equation}
    Then, 
    \begin{equation}\label{equ-result-comparison-BBL}
    u(x,t)\leq v(x,t),\quad\forall~(x,t)\in\mathbb R^n\times [0,T].
    \end{equation} 
\end{theorem}
 Especially, the initial value condition implies that the viscosity subsolution $u$ and supersolution $v$ to \eqref{equ-InitialValProb-Extended} satisfy $u(x,0)\leq v(x,0)$. The uniqueness and comparison principle proved in \cite{Barles-Biton-Ley-UniqueWithoutGrowth} and \cite{Chen-Pang-UniqueLagranMeanFlow} can be summarized into the following, which also holds for upper semicontinuous viscosity subsolution $u$ and lower semicontinuous viscosity supersolution $v$. 

\begin{theorem}[Theorem 2.1 in \cite{Barles-Biton-Ley-UniqueWithoutGrowth} and Theorem 1.1 in \cite{Chen-Pang-UniqueLagranMeanFlow}]\label{thm-Comparison-Entire-BBL}
    Let $T>0$, $f\equiv 0$, and let $u,v\in C^0(\mathbb R^n\times[0,T])$ respectively be  viscosity subsolution and supersolution to \eqref{equ-InitialValProb-Extended}. Then  \eqref{equ-result-comparison-BBL} holds.

    More generally, if  $F:\mathbb R^n\times [0,T]\times\mathbb R^n\times\mathrm{Sym}(n)\rightarrow\mathbb R$ is a continuous function that satisfies 
    
\begin{enumerate}[(I)]
    \item\label{Cond-H1} For any $R>0$, there exists a continuous function $m_R:\mathbb R_+\rightarrow\mathbb R_+$ such that $m_R(0^+)=0$ and 
    \[
        F\left(y,t,\eta(x-y),Y\right)-F\left(x,t,\eta(x-y),X\right)\leqslant m_R\left(\eta|x-y|^2+|x-y|\right),
    \]
    for all $x,y\in\overline{B_R}$ and $t\in [0,T]$, whenever $X,Y\in\mathrm{Sym}(n)$ and $\eta>0$ satisfy 
    \[
        -3\eta\begin{pmatrix}I&0\\0&I\end{pmatrix}\leqslant\begin{pmatrix}X&0\\0&-Y\end{pmatrix}\leqslant3\eta\begin{pmatrix}I&-I\\-I&I\end{pmatrix}.
    \]
    \item\label{Cond-H2} There exist $0<\alpha<1$  and $K_1, K_2>0$ such that 
    \[
        F(x,t,p,X)-F(x,t,q,Y)\leqslant K_1|p-q|\left(1+|x|\right)+K_2\left(\mathrm{tr}(Y-X)^+\right)^\alpha,
    \]
    for every $(x,t,p,X), (x,t,q,Y)\in\mathbb{R}^n\times[0,T]\times\mathbb{R}^n \times \mathrm{Sym}(n).$
    \end{enumerate}     
    Let  $u,v\in C^0(\mathbb R^n\times[0,T])$  respectively be  viscosity subsolution and  supersolution to 
    \[ 
    \left\{
        \begin{array}{llll}
            u_t+F(x,t,Du,D^2u)=0, & \text{in }\mathbb R^n\times (0,T],\\
            u(x,0)=u_0(x), & \text{in }\mathbb R^n.
        \end{array}
    \right.
    \]
    Then the
     comparison principle result \eqref{equ-result-comparison-BBL} holds. 
\end{theorem}

Since $\arctan$ is a smooth, bounded function with gradient no larger than 1, for any $0<\alpha<1$, we have 
  \[
  \arctan b-\arctan a\leq \left\{
    \begin{array}{llll}
        (b-a)^\alpha, & \forall~0<b-a\leq 1,\\
        \pi (b-a)^\alpha, & \forall~1\leq b-a.
    \end{array}
  \right.
  \]
  It follows that for any $0<\alpha<1$, there exists $K>0$ such that 
  \[
  \sum_{i=1}^n\arctan\lambda_i(Y)-\sum_{i=1}^n\arctan\lambda_i(X)\leq
  K(\mathrm{tr}(Y-X)^+)^\alpha,
  \]
  for all  $X,Y\in\mathrm{Sym}(n)$. (See also the proof in \cite{Chen-Pang-UniqueLagranMeanFlow}.) This  enables us to apply the strategy in the proof of Theorem \ref{thm-Comparison-Entire-BBL}.

\begin{lemma}\label{lem-LinearizedEqu-Extended}
    Let $T,f, u, v$ be as in Theorem \ref{thm-Main2-Uniqueness-Extended}. 
    Then $\omega:=u-v\in C^0(\mathbb R^n\times[0,T])$ is a  viscosity subsolution to 
\[
\left\{
    \begin{array}{lllll}
        A[\omega](x,t)= 0, & \text{in }\mathbb R^n\times(0,T),\\ 
        \omega(x,0)\leq 0, & \text{in }\mathbb R^n,
    \end{array}
\right.
\]
where the differential operator $A$ is given by
\[
A[\omega]:=\omega_t-K\left(\mathrm{tr}(D^2\omega)^+\right)^\alpha=0\quad\text{in }\mathbb R^n\times(0,T).
\]
\end{lemma}

\begin{proof}
    Let $\phi\in C^2(\mathbb R^n\times[0,T])$ and $(\bar x,\bar t)\in \mathbb R^n\times(0,T)$ be a strict   maximum point such that 
    \[
    0=\omega(\bar x,\bar t)-\phi(\bar x,\bar t)=\max_{(x,t)\in \mathbb R^n\times[0,T]}(\omega(x,t)-\phi(x,t)).
    \]
    For any sufficiently small $\epsilon>0$, we assume the maximum 
    \[
    \max_{(x,y,t)\in \mathbb R^n\times\mathbb R^n\times[0,T]}
    \left(\omega(x,t)-\phi(x,t)-\frac{|x-y|^2}{\epsilon^2}\right)
    \]
    is achieved at points $(x_\epsilon,y_\epsilon,t_\epsilon)$. By a direct consideration, 
    \[
    (x_\epsilon,y_\epsilon,t_\epsilon)\rightarrow (\bar x,\bar x,\bar t)\quad\text{and}\quad \frac{|x_\epsilon-y_\epsilon|^2}{\epsilon^2}\rightarrow 0\quad\text{as }\epsilon\rightarrow 0.
    \]
    Following the proof of  Lemma 2.1 in \cite{Barles-Biton-Ley-UniqueWithoutGrowth},  there exist $a,b\in\mathbb R$, $X,Y\in\mathrm{Sym}(n)$ such that 
    \[
    (a,D\phi(x_\epsilon,t_\epsilon)+p_\epsilon,X+D^2\phi(x_\epsilon,t_\epsilon))\in\mathcal J^{2,+}u(x_\epsilon,t_\epsilon)\quad\text{and}\quad(b,p_\epsilon,Y)\in\mathcal J^{2,-}v(x_\epsilon,t_\epsilon),
    \]
    where $p_\epsilon:=\frac{2(x_\epsilon-y_\epsilon)}{\epsilon^2}$, $a-b=\phi_t(x_\epsilon,t_\epsilon)$, and $X\leq Y.$ 
    The symbols $\mathcal J^{2,+}u(x_\epsilon,t_\epsilon)$ and $\mathcal J^{2,-}v(x_\epsilon,t_\epsilon)$ are respectively  second order ``superjet'' and ``subjets'' of $u$ and $v$  at $(x_\epsilon,t_\epsilon)$, as defined in \cite{User'sGuide-ViscositySol}.
    Since $u$ and $v$ are, respectively, viscosity subsolution and supersolution to \eqref{equ-InitialValProb-Extended}, we have 
    \[
    \phi_t(x_\epsilon,t_\epsilon)+\sum_{i=1}^n\arctan\lambda_i(Y)-\sum_{i=1}^n\arctan \lambda_i(X)\leq f(x_\epsilon,t_\epsilon)-f(y_\epsilon,t_\epsilon).
    \]
    Therefore, sending $\epsilon\rightarrow 0^+$, it follows that
    \[
    \phi_t(\bar x,\bar t)-K(\mathrm{tr}(D^2\phi)^+)^{\alpha}\leq 0.
    \]
    This finishes the proof that $\omega$ is a viscosity subsolution to $A[\omega]=0$ in $\mathbb R^n\times(0,T)$. 
\end{proof}

 In Lemma \ref{lem-LinearizedEqu-Extended}, if  $u$ and $v$ are classical subsolution and  supersolution respectively, then the result follows immediately from the Newton-Leibnitz formula and the fact that $F(x,t,D^2u):=-\sum_{i=1}^n\arctan\lambda_i(D^2u)-f(x,t)$ satisfies condition \eqref{Cond-H2}.

\begin{lemma}[Lemma 2.2 in \cite{Barles-Biton-Ley-UniqueWithoutGrowth}]\label{lem-FriendlyGiant-Barrier}
    There exist $c>0$ and $k>l>0$ such that for  any $R>0$ sufficiently large, 
    \[
    \chi_R(x,t):=R^l \left(
        (1+R^2)^{\frac{1}{2}}(1-ct)-(1+|x|^2)^{\frac 12}
    \right)^{-k},
    \]
    is a strict supersolution to $A[w]=0$ in the domain
    \[
    \mathcal D(c,R):=\left\{(x,t)\in\mathbb R^n\times [0,T]~:~ct<\frac 12,~
    (1+|x|^2)^{\frac 12}<(1+R^2)^{\frac{1}{2}}(1-ct)
    \right\}.
    \]
\end{lemma}

\begin{proof}[Proof of Theorem \ref{thm-Main2-Uniqueness-Extended}]
The supersolution $\chi_R$ constructed in  Lemma \ref{lem-FriendlyGiant-Barrier} remains positive for all $(x,t)\in \mathcal D(c,R)$ and $\chi_R$ goes to $+\infty$ on the lateral side of $\mathcal D(c,R)$ i.e.,  
  \[
  \chi_R(x,t)\rightarrow\infty,\quad\text{as }(x,t)\rightarrow (x_0,t_0),\quad\forall~(x_0,t_0)\in \partial_p\mathcal D(c,R)\cap (\mathbb R^n\times(0,T_c]),
  \]
  where $\partial_p$ denote the parabolic boundary \cite{Book-Lieberman-SecondParboDE} and  $T_c:=\min\{T,\frac{1}{2c}\}$.
  
  Together with the results in Lemma \ref{lem-LinearizedEqu-Extended},  we may apply comparison principle with respect to $\omega$ and $\chi_R$ on $\mathcal D(c,R)\cap (\mathbb R^n\times (0,T_c])$. It follows that for all
  $(x,t)\in \mathcal D(c,R)\cap (\mathbb R^n\times (0,T_c])$, 
  \[
  (u-v)(x,t)=\omega(x,t)\leq R^l \left(
    (1+R^2)^{\frac{1}{2}}(1-ct)-(1+|x|^2)^{\frac 12}
\right)^{-k}.
  \]
  To conclude, sending  $R\rightarrow\infty$ in the inequality above for $t\in[0,T_c]$, we obtain 
  \[
  u(x,t)\leq v(x,t),\quad\forall~(x,t)\in \mathbb R^n\times[0,T_c].
  \]
  To prove the same property for all $t\in[0,T]$, iterating with respect to time variable  completes the proof of Theorem \ref{thm-Main2-Uniqueness-Extended}. 
\end{proof}

\section{Existence of viscosity solution to Lagrangian mean curvature flow}\label{seclabel-Sec-ExistenceSol}

In this section, we consider the existence of viscosity solution to Lagrangian mean curvature flow on $\mathbb R^{n+1}_-$, with prescribed asymptotic behavior at infinity. 
Theorem \ref{thm-Main3-EntireDiri}, the main result of this section, will serve as a key tool for establishing the asymptotic behavior results stated in Theorems \ref{thm-Main1-AsymBehav} and \ref{thm-main5-GeneralThm}.

\subsection{Translated Dirichlet problems}\label{seclabel-SubSec-TransExterDP}

In this subsection, we apply several necessary transformations to the  Dirichlet type problems as in  
\eqref{equ-Prob-EntireDiri-InThm}.
To start with,
without loss of generality, we assume  $A$ is a diagonal matrix with eigenvalues 
\[
a=(a_1,\cdots,a_n)=\lambda(A),\quad a_1\leq a_2\leq \cdots\leq a_n,
\]
$b=0$ and $c=0$. Otherwise, we take an orthogonal matrix $Q$ such that $Q'AQ=\mathrm{diag}(a_1,\cdots,a_n)$, and consider    $\overline u(x,t):=u(Qx,t)-b\cdot Qx-c$ instead.

Secondly, we take sufficiently large constant $K>0$ and $\Xi>\frac{n}{2}\pi$ such that $A+KI>0$ and $\Xi$ to be determined. Then,  problem \eqref{equ-Prob-EntireDiri-InThm}  is equivalent to the Dirichlet problem corresponding to $\widetilde u(x,t):=u(x,t)+\frac{1}{2}K|x|^2-\Xi t$ i.e., 
\begin{equation}\label{equ-ExteriorDiri-Translated}
    \left\{
        \begin{array}{lllll}
            \displaystyle P[\widetilde u]=\Xi-f(x,t), & \text{in }
            \mathbb R^{n+1}_-,\\
            \displaystyle \widetilde u(x,t)=\widetilde\tau  t+\frac{1}{2}x'Dx+
             O(\mathcal R^{2-\beta}),  & \text{as }\mathcal R(x,t)\rightarrow\infty,
        \end{array}
    \right.
\end{equation}
where  $\widetilde\tau :=\tau-\Xi<0$ and $P$ denote the parabolic operator 
\[
P[\widetilde u]:=- \widetilde u_t+\sum_{i=1}^n\arctan\lambda_i(D^2\widetilde u-KI).
\]
Hereinafter, $D:=A+KI>0$ is a symmetric matrix with eigenvalues $d_1,d_2,\cdots,d_n$.
 
Let 
\[
s(x,t):=\widetilde\tau t+\frac{1}{2}x'Dx=\widetilde\tau t+\frac{1}{2}\sum_{i=1}^nd_ix_i^2,\quad\forall~(x,t)\in\mathbb R^{n+1}_-.
\]
Since $\widetilde\tau <0$ and $d_i> 0$, we have $s(x,t)\geq 0$ on $\mathbb R^{n+1}_-$, with the equality holds only at $(x,t)=(0,0)$, and 
\[
s(x,t)\rightarrow\infty\quad\text{if and only if }-t+|x|^2\rightarrow\infty.
\] 
We say a function $v(x,t)$ is generalized symmetric with respect to $s$, if there exists a scalar function $V:[0,+\infty)\rightarrow\mathbb R$ such that 
\begin{equation}\label{equ-def-GeneraSym}
    v(x,t)=V(s(x,t)),\quad\forall~(x,t)\in\mathbb R^{n+1}_-. 
\end{equation}
When $f$ satisfies condition \eqref{equ-cond-RHS-ConverSpeed}, there exist monotone increasing smooth function $\underline f$ and monotone decreasing smooth function $\overline f$ defined on $[0,+\infty)$ such that 
\begin{equation}\label{equ-choice-fBarrier}
\underline{f}(s(x,t))\leq -f(x,t)\leq \overline f(s(x,t)),\quad\forall~(x,t)\in \mathbb R^{n+1}_-,
\end{equation}
with 
\begin{equation}\label{equ-chioce-fBarrier-converSpeed}
\underline f(s),~\overline f(s)=O(s^{-\frac{\beta}{2}})\quad\text{as }s\rightarrow\infty.
\end{equation}
% When $f\equiv0$, for any $\zeta>0$, we may take $\beta$ larger than $\zeta+2$ and obtain monotone functions $\underline f, \overline f$ satisfying \eqref{equ-choice-fBarrier} and \eqref{equ-chioce-fBarrier-converSpeed}.
Especially, $\underline f(s)<0, \overline f(s)>0$ for all $s\in[0,+\infty)$. Now, we fix $\Xi>\frac n2\pi$ sufficiently large such that
\begin{equation}\label{equ-choice-Theta}
    M(A,\Xi):=
    \left(\Xi-\tau+\sum_{i=1}^n\dfrac{a_i}{1+a_i^2} \right)\left(\sum_{i=1}^n\dfrac{1}{1+a_i^2} \right)^{-1}>\frac{\beta}{2}>1,
\end{equation}
\begin{equation}\label{equ-choice-Theta-1}
    \Xi-\tau +\sum_{i=1}^n\min\{a_i,0\}>0,\quad\text{and}\quad 
    \Xi+\underline f(0)>\frac{n}{2}\pi.
\end{equation}
 
Following the proof of Lemmas 2.7 and 2.8 in \cite{Zhou-Gong-Bao-AncientSolution-paraboicMA}, we obtain the following necessary lemmas of Perron's method for $P[\widetilde u]=\Xi-f(x,t)$. To begin with, we introduce the concepts of weak viscosity sub/supersolutions, 
building on the definitions originally proposed for parabolic Monge-Amp\`ere equations (see \cite{Zhan-PhdThesis-ParboMA,Zhou-Gong-Bao-AncientSolution-paraboicMA} and related literature). 

\begin{definition}[Weak viscosity subsolution and supersolution]
    Let $\Omega\subset\mathbb R^{n+1}$ be an open set in parabolic sense \cite{Book-Lieberman-SecondParboDE}. We say a function $u$ is a weak viscosity subsolution to 
    \begin{equation}\label{equ-lagflow-ondomains}
        P[u](x,t)=\Xi-f(x,t)\quad\text{in }\Omega,
    \end{equation}
    if the upper semicontinuous (USC for short) envelop of $u$ i.e., 
    \[
        u^*(x,t):=\lim_{r\rightarrow 0}\sup_{(x^\sharp,t^\sharp)\in \mathcal Q_r(x,t)}u(x^\sharp,t^\sharp)
    \]
    is finite and a viscosity subsolution to \eqref{equ-lagflow-ondomains}, where 
    \[
       \mathcal Q_r(x,t):=\{(x^\sharp,t^\sharp)\in\Omega~:~|x^\sharp-x|<r,~|t^\sharp-t|<r^2\}.
    \]
    Similarly, one uses lower semicontinuous (LSC for short) envelop $u_*:=-(-u)^*$ for weak viscosity supersolutions. If $u$ is both weak viscosity subsolution and weak viscosity supersolution, it is called a weak viscosity solution.
\end{definition}
 
\begin{lemma}\label{lem-Maintainess-Subsol}
    Let $\Omega\subset\mathbb R^{n+1}_-$ be an open domain in parabolic sense. Let $\mathcal S$ denote any non-empty set of weak viscosity subsolutions to \eqref{equ-lagflow-ondomains}.
    Suppose 
    \[
    u(x,t):=\sup\{v(x,t)~:~v\in\mathcal S\}<\infty,\quad\forall~(x,t)\in\Omega.
    \]
    Then $u$ is a weak viscosity subsolution to \eqref{equ-lagflow-ondomains}.
\end{lemma}
\begin{proof}
    For any $(x,t)\in\Omega$ and $r>0$, let
    \[
        Q_r(x,t):=\{(x^\sharp,t^\sharp)\in \Omega~:~|x^\sharp-x|<r,~t-r^2<t^\sharp\leq t\}.
    \] 
    From the definition of weak viscosity subsolutions, we only need to prove that for all $\varphi\in C^{2,1}(\Omega)$, if for some interior point $(\bar x,\bar t)\in\Omega$ such that 
    \[
    (u^*-\varphi)(\bar x,\bar t)=\max_{\overline{Q_r(\bar x,\bar t)}}(u^*-\varphi),
    \]
    for some $r>0$, then 
    \[
       P[\varphi](\bar x,\bar t)\geq \Xi-f(\bar x,\bar t).
    \] 
    Up to a translation on the value of function, we may assume without loss of generality that  $(u^*-\varphi)(\bar x,\bar t)=0$.
    Take 
    \[
        \psi(x,t)=\varphi(x,t)+|x-\bar{x}|^4+|t-\bar{t}|^2.
    \]
    Then, from the definition of $(\bar x,\bar t)$, $u^*-\psi$ attains its strict maximum point in a neighborhood $\overline{Q_r(\bar x,\bar t)}$. Especially, 
    \[
    (u^*-\psi)(x,t)\leq -(x-\bar x)^4-|t-\bar t|^2,\quad\forall~(x,t)\in Q_r(\bar x,\bar t).
    \]

    From the definition of $u$, for any $k$, there exists $v_k\in\mathcal S$ such that 
    \[
    u(\bar x,\bar t)-\frac{1}{k}<v_k(\bar x,\bar t)\leq u(\bar x,\bar t).
    \]
    Furthermore,
    \[
    (v_k^*-\psi)(x,t)\leq (u^*-\psi)(x,t)\leq -(x-\bar x)^4-|t-\bar t|^2,\quad\forall~(x,t)\in Q_r(\bar x,\bar t).
    \]
    It follows that the maximum of $v_k^*-\psi$ in $Q_r(\bar x,\bar t)$ is attained at some point $(y_k,s_k)$ and
    \[
    -\frac{1}{k}=(u^*-\psi)(\bar x,\bar t)-\frac{1}{k}\leq (v^*_k-\psi)(\bar x,\bar t)\leq (v^*_k-\psi)(y_k,s_k).
    \]
    Combining the estimates above, 
    \[
    -\frac{1}{k}\leq -|y_k-\bar x|^4-|s_k-\bar t|^2\leq 0.
    \]
    Sending $k\rightarrow\infty$, it follows that $(y_k,s_k)\rightarrow (\bar x,\bar t)$
    as $k\rightarrow\infty.$

    Since $v_k$ is a weak viscosity subsolution, it follows that $ P[\psi](y_k,s_k) \geq \Xi-f(y_k,s_k).$
    By the regularity $\psi\in C^{2,1}(\Omega)$, sending $k\rightarrow\infty$ implies $P[\psi](\bar x,\bar t)\geq \Xi-f(\bar x,\bar t)$. 
    At the point $(\bar x,\bar t)$, we have $\psi_t=\varphi_t$ and $D^2\psi=D^2\varphi$. This finishes the proof of the lemma. 
\end{proof}

\begin{lemma}\label{lem-PerronMethod-WeakSol}
    Let $g$ be a weak viscosity supersolution to \eqref{equ-lagflow-ondomains}. Let 
    \[
    \mathcal S_g:=\{v~:~v\text{ is a weak viscosity subsolution to \eqref{equ-lagflow-ondomains} and }v\leq g\},
    \]
    and 
    \[
    u(x,t):=\sup\{v(x,t)~:~v\in\mathcal S_g\},\quad\forall~(x,t)\in\Omega.
    \]
    If $\mathcal S_g$ is not empty, then $u$ is a weak viscosity solution to \eqref{equ-lagflow-ondomains}.
\end{lemma}
\begin{proof}
    By the results in Lemma \ref{lem-Maintainess-Subsol}, $u$ is a weak viscosity subsolution, and it remains to prove it is also a weak viscosity supersolution. Arguing by contradiction, there exist $\varphi\in C^{2,1}(\Omega)$ and $(\bar x,\bar t)\in\Omega$ such that 
    \[
    0=(u_*-\varphi)(\bar x,\bar t)=\min_{\overline{Q_r(\bar x,\bar t)}}(u_*-\varphi),
    \]
    for some $r>0$, and $P[\varphi](\bar x,\bar t)>\Xi-f(\bar x,\bar t)$. 
    We may assume without loss of generality that 
    \[
    (u_*-\varphi)(x,t)\geq |x-\bar x|^4+|t-\bar t|^2,\quad\forall~(x,t)\in Q_r(\bar x,\bar t),
    \]
    otherwise we replace $\varphi$ by $\varphi(x,t)-|x-\bar x|^4-|t-\bar t|^2$ instead.

    Since $u\leq g$, we have $\varphi\leq u_*\leq g_*$ in $Q_r(\bar x,\bar t)$. Furthermore, we claim 
    \[
    u_*(\bar x,\bar t)=\varphi(\bar x,\bar t)<g_*(\bar x,\bar t).
    \]
    Arguing by contradiction, we suppose $u_*(\bar x,\bar t)=\varphi(\bar x,\bar t)=g_*(\bar x,\bar t)$, then $g_*$ was touched by $\varphi$ from below at $(\bar x,\bar t)$, which implies 
    \[
        P[g_*](\bar x,\bar t)\geq P[\varphi](\bar x,\bar t)>\Xi-f(\bar x,\bar t).
    \]
    This contradicts to the assumption that $g$ is a weak viscosity supersolution to \eqref{equ-lagflow-ondomains}.

    Since $f$ is continuous  and $\varphi\in C^{2,1}(\Omega)$, there exists sufficiently small $\delta>0$ such that 
    \[
    P[\varphi](x,t)\geq \Xi-f(x,t)\quad\text{and}\quad
    \varphi(x,t)+\delta^4\leq g_*(x,t),
    \] 
    for $(x,t)\in\overline{\mathcal Q_{2\delta}(\bar x,\bar t)}\subset\Omega$. Then, $\varphi(x,t)+\delta^4$ is a subsolution in $\mathcal Q_{2\delta}$ and 
    \[
    u(x,t)\geq u_*(x,t)\geq \varphi(x,t)+|x-\bar x|^4+|t-\bar t|^2\geq \varphi(x,t)+\delta^4\quad\text{in }\mathcal Q_{2\delta}\setminus \mathcal Q_\delta(\bar x,\bar t).
    \]
    Let 
    \[
    w(x,t):=\left\{
        \begin{array}{llllll}
            \max\{\varphi(x,t)+\delta^2,u(x,t)\}, & (x,t)\in \mathcal Q_{\delta}(\bar x,\bar t),\\
            u(x,t), & (x,t)\in\Omega\setminus \mathcal Q_\delta(\bar x,\bar t).
        \end{array}
    \right.
    \]
    Then $w$ remains a weak viscosity subsolution to \eqref{equ-lagflow-ondomains}. From the definition of $u$ and the fact that $w\leq g$, we have $u\geq w.$

    On the other hand, 
    \[
    0=(u_*-\varphi)(\bar x,\bar t)=\lim_{r\rightarrow 0}\inf_{(x^\sharp,t^\sharp)\in \mathcal Q_r(\bar x,\bar t)}(u-\varphi)(x^\sharp,t^\sharp).
    \]
    Hence, there exist a point $(x^\sharp,t^\sharp)\in \mathcal Q_{\delta}(\bar x,\bar t)$ with $\delta>0$ sufficiently small such that 
    \[
    u(x^\sharp,t^\sharp)-\varphi(x^\sharp,t^\sharp)<\delta^4,\quad\text{i.e.,}\quad 
    u(x^\sharp,t^\sharp)<\varphi(x^\sharp,t^\sharp)+\delta^4\leq w(x^\sharp,t^\sharp).
    \]
    This becomes a contradiction and finishes the proof of the desired result. 
\end{proof}

 \subsection{Construction of subsolutions}\label{seclabel-Subsec-ConstrucSubsol}

 \begin{lemma}\label{lem-Subsol-EigenValEst}
   Let $V\in C^2([0,+\infty))$, $V'\geq 1, V''\leq 0$ and $v$ be as in \eqref{equ-def-GeneraSym}.
   Then $v\in C^{2,1}(\mathbb R^{n+1}_-)$ is   generalized symmetric with respect to $s$, and  satisfies
   \begin{equation}\label{equ-Est-EigenVal-Subsol}
   d_iV'(s)+\sum_{j=1}^nd_j^2x_j^2V''(s)\leq\lambda_i(D^2v(x,t))\leq d_iV'(s),\quad\forall~1\leq i\leq n.
   \end{equation}
   Furthermore, 
   \[
   P[v](x,t)\geq -\widetilde\tau V'+\sum_{i=1}^n\arctan\left(a_iV'+\sum_{j=1}^nd_j^2x_j^2V''\right),\quad\forall~(x,t)\in\mathbb R^{n+1}_-.
   \]
 \end{lemma}
 \begin{proof}
    By a direct computation,  for all $(x,t)\in\mathbb R^{n+1}_-\setminus\{(0,0)\}$, 
    \[
    \begin{array}{lllll}
        -v_t(x,t) &=&-\widetilde\tau V'(s),\\
      D_iv(x,t) &=&d_ix_iV'(s),\\
      D_{ij}v(x,t) &=&d_i\delta_{ij}V'(s)+d_id_jx_ix_jV''(s),
    \end{array}
    \] 
  where $\delta_{ij}$ denotes the Kronecker delta symbol. Consequently, the regularity of $V$ implies $v\in C^{2,1}(\mathbb R^{n+1}_-\setminus\{(0,0)\})$. Then, as $(x,t)\rightarrow(0,0)$,
  \[
  v_t(x,t)\rightarrow \widetilde\tau V'(0),\quad D_iv(x,t)\rightarrow 0,\quad\text{and}\quad 
  D_{ij}v(x,t)\rightarrow d_i\delta_{ij}V'(0).
  \]
  By the mean value theorem, this finishes the proof that $v\in C^{2,1}(\mathbb R^{n+1}_-)$. 
    
   Rewrite the Hessian matrix $D^2v$ into
     \[
     D^2v(x,t) =\mathfrak{A}(x,t)+\mathfrak{B}(x,t),
     \]
     where 
     $\mathfrak{A}=\mathrm{diag}(d_1V'(s),\cdots,d_nV'(s))$ and 
     $\mathfrak{B}=(d_id_jx_ix_jV''(s))_{n\times n}$.
     By a direct computation, 
     \[
     \lambda(\mathfrak A)=(d_1V'(s),\cdots,d_nV'(s))\quad\text{and}\quad
     \lambda(\mathfrak B)=\left(
         \sum_{j=1}^nd_j^2x_j^2V''(s),0,\cdots,0
     \right).
     \]
     Since $V''\leq 0$ and $V'\geq 1$, estimate  \eqref{equ-Est-EigenVal-Subsol} follows immediately. (See also \cite{Bao-Liu-Wang-DiricLagrangianEntire,Li-Wang-EntireDirichlet,Li-Dirichlet-SPL}, etc.)
     Therefore, by the monotonicity of $\arctan$ function, it follows that 
     \[
     \begin{array}{llll}
         P[v](x,t)&\geq &\displaystyle -\widetilde\tau V'+\sum_{i=1}^n\arctan \left(d_iV'+\sum_{j=1}^nd_j^2x_j^2V''-K\right)\\
         &\geq &\displaystyle -\widetilde\tau V'+\sum_{i=1}^n\arctan \left(d_iV'+\sum_{j=1}^nd_j^2x_j^2V''-KV'\right)\\
         &= & \displaystyle -\widetilde\tau V'+\sum_{i=1}^n\arctan\left(a_iV'+\sum_{j=1}^nd_j^2x_j^2V''\right).
     \end{array}
     \]
     This finishes the proof of the desired estimate. 
 \end{proof}
 
 To proceed, we introduce the following two implicit functions. 
 \begin{lemma}\label{lem-construction-ImplicFunc-Hsw-F1}
     Let $\overline f$ be the monotone decreasing function as chosen in \eqref{equ-choice-fBarrier} and \eqref{equ-chioce-fBarrier-converSpeed}. Then there exists a unique decreasing positive function $\underline w(s)$ defined on $[0,+\infty)$ determined by 
     \[
     F_1(\underline w(s)):=-\widetilde\tau \underline w(s)+\sum_{i=1}^n\arctan (a_i\underline w(s))=\Xi+\overline f(s),\quad\forall~s\geq 0.
     \]
     Especially, $\underline w(s)>1$ and satisfies $\underline w(s)\rightarrow 1$ as $s\rightarrow\infty$.
 \end{lemma}
 \begin{proof}
     From the choice of $\Xi$ as in \eqref{equ-choice-Theta} and \eqref{equ-choice-Theta-1}, we have 
     \[
     F_1'(\underline w)=-\widetilde\tau +\sum_{i=1}^n\dfrac{a_i}{1+a_i^2\underline w^2}\geq 
     \Xi-\tau +\sum_{i=1}^n\min\{a_i,0\}>
     0,
     \]
     for all $\underline w\geq 1.$ Furthermore, 
     \[
     F_1(1)=-\widetilde\tau +\sum_{i=1}^n\arctan a_i=\Xi< \Xi+\overline f(s),\quad\forall~s\geq 0,
     \]
     and $F_1(\underline w)\rightarrow +\infty$ as $\underline w\rightarrow\infty$. 
     Hence, by the mean value theorem, there exists a unique $\underline w(s)>1$ such that $F_1(\underline w(s))=\Xi+\overline f(s)$.  Sending $s\rightarrow\infty$, it follows immediately that $\underline w(s)\rightarrow 1$ at infinity. By taking partial derivative with respect to $s$, 
     \[
     -\widetilde\tau \underline w'(s)+\sum_{i=1}^n\dfrac{a_i}{1+(a_i\underline w(s))^2}\underline w'(s)=\overline f'(s)<0.
     \]
     Thus, $\underline w(s)$ is a monotone decreasing function. 
 \end{proof}
 
 Furthermore, there exists a unique decreasing positive function $w^*(s)$ defined on $[0,+\infty)$ determined by 
 \[
 -\widetilde\tau w^*(s)-\frac{n}{2}\pi=\Xi+\overline f(s),\quad\forall~s\geq 0.
 \]
 Especially, from the choice of $\underline w(s)$ in Lemma \ref{lem-construction-ImplicFunc-Hsw-F1}, $w^*(s)>\underline w(s)$, and 
 \[
 \lim_{s\rightarrow\infty}w^*(s)=\frac{\Xi+\frac{n}{2}\pi}{-\widetilde\tau }=\frac{\Xi+\frac{n}{2}\pi}{\Xi-\tau}>1,
 \]
 since $\tau=\sum_{i=1}^n\arctan\lambda_i(A)\in \left(-\frac{n}{2}\pi,\frac{n}{2}\pi\right)$. Thus, there exist $S, \delta>0$ such that 
 \[
 w^*(s)\geq \frac{\underline w(\infty)+w^*(\infty)}{2}+\delta,
 \quad \text{and}\quad 
 \frac{\underline w(\infty)+w^*(\infty)}{2}\geq \underline w(s),\quad\forall~s\geq S. 
 \]

 \begin{lemma}\label{lem-construction-ImplicFunc-Hsw}
     There exists a   smooth function $h(s,w)\leq 0$  satisfying 
     \[
     F_2(w,h):=-\widetilde\tau w+\sum_{i=1}^n\arctan\left(a_iw+ h\right)=\Xi+\overline f(s),\quad\forall~ w\in[\underline w(s),w^*(s)),\quad s\geq 0.
     \]
      Furthermore, $h(s,w)$ is monotone decreasing with respect to $w$ and to $s$,
     $h(s,\underline w(s))\equiv 0,  
     h(\infty,1)=0$,
     \[
      \dfrac{\partial h}{\partial w}(\infty,1)=-M(A,\Xi),
     \]
     and 
     \begin{equation}\label{equ-AsymBehav-hsw-Subsol}
     |h(s,w)-h(\infty,w)|\leq Cs^{-\frac{\beta}{2}},\quad\forall~w\in\left[\underline w(s),\frac{\underline w(\infty)+w^*(\infty)}{2}\right],
     \end{equation}
     for sufficiently large $s$ and $C>0$.
 \end{lemma}
 \begin{proof}
     Notice that for all $w\in[\underline w(s),w^*(s))$, the monotonicity of $\arctan$ function implies that  
     \[
     \lim_{h\rightarrow 0}F_2(w,h)=F_2(w,0)=-\widetilde\tau w+\sum_{i=1}^n\arctan(a_iw)\geq \Xi+\overline f(s),\quad\forall~s\geq 0,
     \]
     and 
     \[
     \lim_{h\rightarrow-\infty}F_2(w,h)=-\widetilde\tau w-\frac{n}{2}\pi<\Xi+\overline f(s),\quad\forall~s\geq 0.
     \]
     Since $F_2$ is monotone increasing with respect to $h$, the mean value theorem proves that there exists a unique implicit function $h(s,w)$ such that 
     \[
     F_2(w,h(s,w))=\Xi+\overline f(s),\quad\forall~w\in[\underline w(s),w^*(s)),~s\geq 0. 
     \]
     Especially, $-\infty<h(s,w)\leq 0$ and is bounded when $w$ is away from $w^*(s)$. By the implicit function theorem, $h$ is a smooth function that satisfies 
     \[
     \dfrac{\partial F_2}{\partial h}(w,h)\cdot \dfrac{\partial h}{\partial w}(s,w)+\dfrac{\partial F_2}{\partial w}=0,\quad\text{and}\quad    \dfrac{\partial F_2}{\partial h}(w,h)\cdot \dfrac{\partial h}{\partial s}(s,w)=\overline f'(s)<0.
     \]
     By a direct computation, together with the choice of $\Xi$ as in \eqref{equ-choice-Theta} and \eqref{equ-choice-Theta-1},
     the  inequalities above imply that $h(s,w)$ is monotone decreasing with respect to $w$ and  $s$. The first inequality proves
     \[
     \sum_{i=1}^n\dfrac{1}{1+(a_iw+h)^2} \cdot \dfrac{\partial h}{\partial w}(s,w)-\widetilde\tau +\sum_{i=1}^n\dfrac{a_i}{1+(a_iw+ h)^2} =0.
     \] 
     Consequently, sending $(s,w)\rightarrow (\infty,1)$, we have $h(s,w)\rightarrow 1$, and 
     \[
     \dfrac{\partial h}{\partial w}(\infty,1)=\left(\widetilde\tau -\sum_{i=1}^n\dfrac{a_i}{1+a_i^2}\right)\cdot \left(\sum_{i=1}^n\dfrac{1}{1+a_i^2}\right)^{-1}=-M(A,\Xi).
     \]
     Eventually, for any $w\in\left[\underline w(s),\frac{\underline w(\infty)+w^*(\infty)}{2}\right]$ and sufficiently large $s$, $\underline w(s), w^*(s)$ and $h(s,w)$ are bounded functions. Then, there exists $C>0$ such that 
     \[
     \begin{array}{llllll}
         \overline f(s)-0&=& F_2(w,h(s,w))-F_2(w,h(\infty,w))\\
         &=&\displaystyle \sum_{i=1}^n\left(\arctan(a_iw+h(s,w))-\arctan(a_iw+h(\infty,w))\right)\\
         &\geq & C(h(s,w)-h(\infty,w)).
     \end{array}
     \]
     Hence, \eqref{equ-AsymBehav-hsw-Subsol} follows from the asymptotic behavior of $\overline f$. 
     This finishes the proof of this lemma. 
 \end{proof}

 \begin{corollary}\label{Coro-SubsolODE} Let  $v$ be as in \eqref{equ-def-GeneraSym}. If $V\in C^2([0,+\infty))$,
     \[
     V'\geq 1,\quad V''\leq 0\quad\text{and}\quad 
     2d_n(s+1)V''\geq h(s,V')\quad\text{in }s\geq 0.
     \]
     Then $v$ is a subsolution to $P[v]=\Xi-f(x,t)$ in $\mathbb R^{n+1}_-$.
 \end{corollary}
 \begin{proof}
     From the definition of generalized symmetric function as in \eqref{equ-def-GeneraSym}, by Lemmas \ref{lem-Subsol-EigenValEst} and \ref{lem-construction-ImplicFunc-Hsw}, 
     \[
     \begin{array}{llllll}
         P[v]&\geq &\displaystyle  -\widetilde\tau V'+\sum_{i=1}^n\arctan\left(a_iV'+\sum_{i=1}^nd_j^2x_j^2V''\right)\\
         &\geq & \displaystyle -\widetilde\tau V'+\sum_{i=1}^n\arctan\left(
             a_iV'+2d_n(s+1)V''
         \right)\\
         &\geq & F_2(V',h(s,V'))\\
         &= & \Xi+\overline f(s),
     \end{array}
     \]
     in $\mathbb R^{n+1}_-$. This finishes the proof of this lemma. 
 \end{proof}
 
 To proceed, we prove the following existence and uniqueness result on an initial value problem of ordinary differential equation, along with the monotonicity and asymptotic behavior at infinity. This strategy has now become  standard in the analysis of Dirichlet problems on exterior domains and entire spaces. See for instance  \cite{Bao-Liu-Wang-DiricLagrangianEntire,Bao-Liu-Zhou-AncientSol-NewAsym-JDE}.

 \begin{lemma}\label{lem-ODEInitial-Subsol-ConsRHS}
     For any   $w_0\in (\underline w(0),w^*(0))$, there exists a unique solution $W\in C^1([0,+\infty))$ to the initial value problem 
     \begin{equation}\label{equ-ODEInitial-Subsol-ConsRHS}
         \left\{
             \begin{array}{llll}
                 \dfrac{\mathrm dw}{\mathrm ds}=\dfrac{h(s,w)}{2d_n(s+1)}, & \text{in }s\geq 0,\\
                 w(0)=w_0.
             \end{array}
         \right.
     \end{equation}
     Furthermore, $W'(s)<0$, $\underline w(s)<W(s)<w^*(s)$, and 
     \begin{equation}\label{equ-Asym-InitSols}
         W(s)=1+O(s^{-\frac{\beta}{2}}),\quad\text{as }s\rightarrow\infty.
     \end{equation}
 \end{lemma}
 
 \begin{proof} \textbf{Step 1. }Local existence near $s=0$. 
     Since $w_0\in (\underline w(0),w^*(0))$,  there exists $s_0>0$ such that 
     \[
     [0,s_0]\times\left[\dfrac{w_0+\underline w(0)}{2},\dfrac{3w_0-\underline w(0)}{2}\right]\subset\{(s,w)~:~s\geq 0,~w\in(\underline w(s),w^*(s))\}.
     \]
     Thus, the right-hand side term $\frac{h(s,w)}{2d_n(s+1)}$ is Lipschitz in the rectangle $[0,s_0]\times\left[\frac{w_0+\underline w(0)}{2},\frac{3w_0-\underline w(0)}{2}\right]$. By the Picard-Lindel\"of theorem, the initial value problem \eqref{equ-ODEInitial-Subsol-ConsRHS} admits a unique solution $W$ near $s=0$, and we shall prove that the solution can be extended to $s\in[0,+\infty)$. 
     
     \textbf{Step 2. }Upper and lower bound estimate.
     By Lemma \ref{lem-construction-ImplicFunc-Hsw} and the choice of $h(s,w)$, the equation in \eqref{equ-ODEInitial-Subsol-ConsRHS} implies that $W$ is monotone decreasing as long as $W>\underline w$. Next, since $W(0)=w_0\in(\underline w(0),w^*(0))$, we claim that $W$ cannot touch $\underline w(s)$ from above. Arguing by contradiction, we suppose there exists $s_1>0$ such that 
     \[
     W(s_1)=\underline w(s_1)\quad\text{and}\quad W(s)>\underline w(s)\quad\text{in }[0,s_1).
     \]
     Then, from the definition of derivative and the equation in \eqref{equ-ODEInitial-Subsol-ConsRHS}, 
     \[
     W'(s_1)=\lim_{s\rightarrow s_1^-}\dfrac{W(s)-W(s_1)}{s-s_1}\leq \underline w'(s_1)\quad\text{but}\quad W'(s_1)=\dfrac{h(s_1,W(s_1))}{2d_n(s_1+1)}=0,
     \]
     where we used the   the fact that $h(s_1,\underline w(s_1))=0$ from Lemma \ref{lem-construction-ImplicFunc-Hsw}. This becomes a contradiction since $\underline w$ is a monotone decreasing function. 
     
     Similarly, we claim that $W$ cannot touch $w^*(s)$ from below. Arguing by contradiction, we suppose there exists $s_2>0$ such that 
     \[
     W(s_2)=w^*(s_2)\quad\text{and}\quad W(s)<w^*(s)\quad\text{in }[0,s_2).
     \]
     Then, from the definition of derivative and the equation in \eqref{equ-ODEInitial-Subsol-ConsRHS} implies 
     \[
     W'(s_2)=\lim_{s\rightarrow s_2^-}\dfrac{W(s)-W(s_2)}{s-s_2}\geq (w^*)'(s_2)\quad\text{but}\quad \lim_{s\rightarrow s_2^-}W'(s)=
     \lim_{s\rightarrow s_2^-}\dfrac{h(s,W(s))}{2d_n(s+1)}=-\infty,
    \]
    where we used the fact that $h(s,w^*(s))=\infty$ from the proof of Lemma \ref{lem-construction-ImplicFunc-Hsw}. This becomes a contradiction since $w^*$ is a smooth function.
     Combining the results above, $W(s)\in(\underline w(s),w^*(s))$ and is a  decreasing function. By the Carath\'eodory extension theorem of ordinary differential equations, the solution $W$ exists on $[0,+\infty)$.
     
     \textbf{Step 3. }Asymptotic behavior of $W$ at infinity. Since $W$ is monotone decreasing and bounded by $1\leq \underline w(s)< w^*(s)$, $W$ admits a limit $W(\infty)\in[1,w^*(\infty)]$ at infinity. We claim that $W(\infty)=1$. Arguing by contradiction, we suppose $W(\infty)>1$. From the results in Lemmas \ref{lem-construction-ImplicFunc-Hsw-F1} and \ref{lem-construction-ImplicFunc-Hsw}, $W(\infty)-\underline w(\infty)>0$ and hence there exists $c>0$ such that 
     \[
     W(s)-\underline w(s)\geq c,\quad\forall~s\geq 1.
     \]
     Since $h(s,\underline w(s))\equiv0$ and $h(s,w)$ is monotone decreasing with respect to $w$ and $s$, we have 
     \[
     \dfrac{\mathrm d W(s)}{\mathrm ds}=\dfrac{h(s,W(s))}{2d_n(s+1)}= \dfrac{h(s,W(s))-h(s,\underline w(s))}{4d_ns}<-\dfrac{\epsilon}{s},\quad\forall~s\gg 1,
     \]
     where $\epsilon$ is a positive constant relying on $c$ and $h(s,w)$. This contradicts the fact that $W$ converge to $W(\infty)$ at infinity, since $-\frac{\epsilon}{s}$ is not integrable on $[1,\infty)$. To proceed, we refine the asymptotic behavior by setting 
     \[
     t:=\ln(s+1)\in (0,+\infty)\quad\text{and}\quad \varphi(t):=W(s(t))-1.
     \]
     By a direct computation, for all $t\in (0,+\infty)$, 
     \[
     \begin{array}{llll}
       \varphi'(t)&=& W'(s(t))\cdot e^t\\
       &=& \dfrac{h(s(t),\varphi+1)}{2d_n}\\
       &=& \dfrac{h(s(t),\varphi+1)-h(\infty,\varphi+1)}{2d_n}
       +\dfrac{h(\infty,\varphi+1)}{2d_n}\\
       &=:& h_1(t,\varphi)+h_2(\varphi).
     \end{array}
     \]
     By the asymptotic behavior of $h(s,w)$ given in Lemma \ref{lem-construction-ImplicFunc-Hsw}, there exists $C>0$ such that for all $t\gg 1$ and $0< \varphi\ll1$, 
     \[
     |h_1(t,\varphi)|\leq Ce^{-\frac{\beta}{2}t},\quad h_2'(0)=\dfrac{\partial h}{\partial w}(\infty,1)=-M(A,\Xi)\quad\text{and}\quad\left|
     h_2(\varphi)-h_2'(0)\cdot\varphi
     \right|\leq C\varphi^2.
     \]
     Consequently, $\varphi$ satisfies 
     \begin{equation}\label{equ-temp-ODE-varphi-subsol}
     \varphi'(t)=-M(A,\Xi)\varphi+O(e^{-\frac{\beta}{2}t})+O(\varphi^2),
     \end{equation}
     as $t\rightarrow\infty$ and $\varphi\rightarrow 0$. By the asymptotic stability result of ordinary differential equations (see for instance Theorem 1.1 in \cite[Chap. 13]{Book-Coddington-Levinson-ODE} or Theorem 2.16 in \cite{Book-Bodine-Lutz-AsymptoticIntegration}), together with the fact that $M(A,\Xi)>\frac{\beta}{2}$, we have 
     \begin{equation}\label{equ-desiredEst-Subsol-varphi}
     \varphi=O(e^{-\frac{\beta}{2}t})\quad\text{as }t\rightarrow\infty.
     \end{equation}
     More explicitly, since $\varphi>0$,  choose sufficiently small $0<\epsilon< M(A,\Xi)-\frac{\beta}{2} $. Then,  $\varphi$ satisfies 
     \[
     \varphi'(t)\leq -(M(A,\Xi)-\epsilon)\varphi(t)+Ce^{-\frac{\beta}{2}t},\quad\forall~t>T_0,
     \]
     for some constants $C, T_0$ sufficiently large. Multiplying both sides by $e^{(M(A,\Xi)-\epsilon)t}$ and integrating over $(T_0,t)$, there exist constants $C>0$ such that 
     \[
     \varphi\leq Ce^{-(M(A,\Xi)-\epsilon)t}+Ce^{-\frac{\beta}{2}t} \leq  Ce^{-\frac{\beta}{2}t},\quad\forall~t\geq T_0.
     \]
     This finishes the proof of  estimate \eqref{equ-desiredEst-Subsol-varphi}, which leads to the desired asymptotic behavior of $W$. 
 \end{proof}

 For any $w_0\in(\underline w(0),w^*(0))$, let 
 \[
 \underline v(x,t):=V_-(s(x,t)):=\int_{0}^{s(x,t)}W(\tau)\mathrm d\tau+C,\quad\forall~(x,t)\in\mathbb R^{n+1}_-,
 \]
 where $W(\tau)$ is the solution to  problem \eqref{equ-ODEInitial-Subsol-ConsRHS} with  initial value $W(0)=w_0$, constructed in Lemma \ref{lem-ODEInitial-Subsol-ConsRHS} and $C$ is a constant to be determined.  
 By Lemmas \ref{lem-Subsol-EigenValEst}, \ref{lem-ODEInitial-Subsol-ConsRHS}, and Corollary \ref{Coro-SubsolODE}, $\underline v\in C^2(\mathbb R^{n+1}_-)$ satisfies 
 \[
 P[\underline v](x,t)\geq \Xi-f(x,t)\quad\text{in }\mathbb R^{n+1}_-.
 \]
 It remains to reveal the asymptotic behavior of $\underline v$ at infinity.

 \begin{lemma}\label{lem-ConsTru-Subsol-Exterio-Trans}
     There exists a constant  $C$ such that the function $\underline v(x,t)$ defined above satisfies
     \[
     \left\{
         \begin{array}{lllll}
             P[\underline v]\geq \Xi-f(x,t), & \text{in }\mathbb R^{n+1}_-,\\ 
             \displaystyle \underline v(x,t)=
             \widetilde\tau t+\frac{1}{2}x'Dx+O(s^{1-\frac{\beta}{2}}(x,t)), & \text{as }-t+|x|^2\rightarrow\infty.
         \end{array}
     \right.
     \]
 \end{lemma}
 \begin{proof}
     From the definition of $\underline v$ and the asymptotic behavior \eqref{equ-Asym-InitSols} satisfied by $W(s)$ given in Lemma \ref{lem-ODEInitial-Subsol-ConsRHS}, we have 
     \[
        \lim_{-t+|x|^2\rightarrow\infty}(\underline v(x,t)-s(x,t))
        =\lim_{s\rightarrow\infty}
        \int_0^s\left(W(\tau)-1\right)\mathrm d \tau+C.
     \]
     By  \eqref{equ-Asym-InitSols}, we may take 
     \[
     C:=-\int_0^\infty\left(W(s)-1\right)\mathrm d s.
     \]
     Therefore, 
     \[
     \underline v(x,t)-s(x,t)=O(s^{1-\frac{\beta}{2}}(x,t))\quad\text{as }-t+|x|^2\rightarrow\infty.
     \]
     This finishes the construction of ancient subsolution on entire space $\mathbb R^{n+1}_-$.
 \end{proof}

 \subsection{Construction of supersolutions}\label{seclabel-Subsec-ConstrucSupsol}

 \begin{lemma}\label{lem-Supsol-EigenValEst}
    Let $V\in C^2([0,+\infty))$, $0<V'\leq 1, V''\geq 0$ and $v$ be as in \eqref{equ-def-GeneraSym}.
   Then $v\in C^{2,1}(\mathbb R^{n+1}_-)$ is   generalized symmetric with respect to $s$, and  satisfies
   \[
   d_iV'(s)\leq\lambda_i(D^2v(x,t))\leq d_iV'(s)+\sum_{j=1}^nd_j^2x_j^2V''(s),\quad\forall~1\leq i\leq n.
   \]
   Furthermore,  
   \[
   P[v](x,t)\leq -\widetilde\tau V'+\sum_{i=1}^n\arctan\left(a_iV'+\sum_{j=1}^nd_j^2x_j^2V''\right),\quad\forall~(x,t)\in\mathbb R^{n+1}_-.
   \]
 \end{lemma} 
 \begin{proof}
    The regularity $v\in C^{2,1}(\mathbb R^{n+1}_-)$ and the estimates of $\lambda_i(D^2v(x,t))$ follow from the proof as in Lemma   \ref{lem-Subsol-EigenValEst}. The conditions $V''\geq 0$ and $0<V'\leq 1$ imply that 
     \[
     \begin{array}{llll}
         P[v](x,t)&\leq &\displaystyle -\widetilde\tau V'+\sum_{i=1}^n\arctan \left(d_iV'+\sum_{j=1}^nd_j^2x_j^2V''-K\right)\\
         &\leq & \displaystyle -\widetilde\tau V'+\sum_{i=1}^n\arctan\left(d_iV'+\sum_{j=1}^nd_j^2x_j^2V''-KV'\right)\\
         &= & \displaystyle -\widetilde\tau V'+\sum_{i=1}^n\arctan\left(a_iV'+\sum_{i=1}^nd_j^2x_j^2V''\right).
     \end{array}
     \]
     This finishes the proof of the desired estimate. 
 \end{proof}
 
 To proceed, we introduce the following two implicit functions. 
 
 \begin{lemma}\label{lem-construction-ImplicFunc-Hsw-F3}
     Let $\underline f$ be the monotone increasing function as chosen in \eqref{equ-choice-fBarrier} and \eqref{equ-chioce-fBarrier-converSpeed}. Then there exists a unique increasing positive function $\overline w(s)$ defined on $[0,+\infty)$ determined by 
     \[
     F_1(\overline w(s))=-\widetilde\tau \overline w(s)+\sum_{i=1}^n\arctan(a_i\overline w(s))=
     \Xi+\underline f(s),\quad\forall~s\geq 0.
     \]
     Especially, $0<\overline w(s)<1$ and satisfies $\overline w(s)\rightarrow 1$ as $s\rightarrow\infty.$
 \end{lemma}
 
 \begin{proof}
     From the choice of $\Xi$ as in \eqref{equ-choice-Theta} and \eqref{equ-choice-Theta-1}, we have 
     \[
     F_1'(\overline w)=-\widetilde\tau +\sum_{i=1}^n\dfrac{a_i}{1+a_i^2\overline w^2}>
     \Xi-\tau +\sum_{i=1}^n\min\{a_i,0\}>
     0,
     \]
     for all $0\leq \overline w\leq 1$. Furthermore,  
     \[
     F_1(0)=0<\Xi+\underline f(s)\quad\text{and}\quad F_1(1)=-\widetilde\tau +\sum_{i=1}^n\arctan a_i=\Xi > \Xi+\underline f(s).
     \]
     Hence, by the mean value theorem, there exists a unique $0<\overline w(s)< 1$ such that $F_1(\overline w(s))=\Xi+\underline f(s)$.  Sending $s\rightarrow\infty$, it follows immediately that $\overline w(s)\rightarrow 1$ at infinity. By taking partial derivative with respect to $s$, 
     \[
     -\widetilde\tau \underline w'(s)+\sum_{i=1}^n\dfrac{a_i}{1+(a_i\overline w(s))^2}\overline w'(s)=\underline f'(s)>0.
     \]
     Thus, $\overline w(s)$ is a monotone increasing function. 
 \end{proof}
 
 Furthermore, there exists a unique increasing  function $w_*(s)$   on $[0,+\infty)$ such that 
 \[
 -\widetilde\tau w_*(s)+\frac{n}{2}\pi=\Xi+\underline f(s),\quad\forall~s\geq 0.
 \]
 Especially, from the choice of $\overline w(s)$ in Lemma \ref{lem-construction-ImplicFunc-Hsw-F3} and \eqref{equ-choice-Theta-1}, $0<w_*(s)<\overline w(s)$, and 
 \[
 \lim_{s\rightarrow\infty}w_*(s)=\dfrac{\Xi-\frac{n}{2}\pi}{-\widetilde\tau }=\dfrac{\Xi-\frac{n}{2}\pi}{\Xi-\tau}<1,
 \]
 since $\tau=\sum_{i=1}^n\arctan\lambda_i(A)\in\left(-\frac{n}{2}\pi,\frac{n}{2}\pi\right)$. Thus, there exist $S,\delta>0$ such that 
 \[
 \overline w(s)\geq \frac{w_*(\infty)+\overline w(\infty)}{2}\quad\text{and}\quad \frac{w_*(\infty)+\overline w(\infty)}{2}\geq 
 w_*(s)+\delta,\quad\forall~s\geq S.
 \]
 
 \begin{lemma}\label{lem-InverseFunc-Supsol}
     There exists a   smooth function $h(s,w)\geq 0$ satisfying 
     \[
     F_2(w,h)=-\widetilde\tau w+\sum_{i=1}^n\arctan\left(a_iw+ h\right)=\Xi+\underline f(s),\quad\forall~w\in(w_*(s),\overline w(s)],~s\geq 0.
     \]
     Furthermore, $h(s,w)$ is monotone decreasing with respect to $w$ and monotone increasing with respect to $s$, 
     $h(s,\overline w(s))\equiv 0$,
     \[
     \dfrac{\partial h}{\partial w}(\infty,1)=-M(A,\Xi),
     \]
     and 
     \begin{equation}\label{equ-AsymBehav-hsw-Supsol}
     |h(s,w)-h(\infty,w)|\leq Cs^{-\frac{\beta}{2}},\quad\forall~ w\in\left[\frac{w_*(\infty)+\overline{w}(\infty)}{2},\overline{w}(s)\right],
     \end{equation}
     for sufficiently large $s$ and $C>0$. 
 \end{lemma}
 \begin{proof}
     Notice that for all $w\in(w_*(s),\overline w(s)]$, we have 
     \[
     \lim_{h\rightarrow 0}F_2(w,h)=F_2(w,0)=-\widetilde\tau w+\sum_{i=1}^n\arctan\left(a_iw\right)
     \leq \Xi+\underline f(s),\quad\forall~s\geq 0,
     \]
     and 
     \[
     \lim_{h\rightarrow+\infty}F_2(w,h)=-\widetilde\tau w+\frac{n}{2}\pi>\Xi
     \geq \Xi+\underline f(s),\quad\forall~s\geq 0.
     \]
     Since $F_2$ is monotone increasing with respect to $h$, the mean value theorem proves that there exists a unique implicit function $h(s,w)$ such that 
     \[
     F_2(w,h(s,w))=\Xi+\underline f(s),\quad\forall~w\in(w_*(s),\overline w(s)],~s\geq 0.
     \]
     Especially, $0\leq h(s,w)<+\infty$ and is bounded when $w$ is away from $w_*(s)$. By the implicit function theorem, $h$ is a smooth function that satisfies 
     \[
     \dfrac{\partial F_2}{\partial h}(w,h)\cdot \dfrac{\partial h}{\partial w}(s,w)+\dfrac{\partial F_2}{\partial w}=0,
     \quad\text{and}\quad 
      \dfrac{\partial F_2}{\partial h}(w,h)\cdot \dfrac{\partial h}{\partial s}(s,w)=\underline f'(s)>0.
     \]
     By a direct computation, together with the choice of $\Xi$ as in \eqref{equ-choice-Theta} and \eqref{equ-choice-Theta-1}, the  inequalities above imply that $h(s,w)$ is monotone decreasing with respect to $w$ and increasing with respect to $s$. Sending $(s,w)\rightarrow (\infty,1)$, we have $h(s,w)\rightarrow 1$ and
     \[
     \dfrac{\partial h}{\partial w}(\infty,1)=\left(\widetilde\tau -\sum_{i=1}^n\dfrac{a_i}{1+a_i^2}\right)\cdot \left(\sum_{i=1}^n\dfrac{1}{1+a_i^2}\right)^{-1}=-M(A,\Xi).
     \]
     Eventually, for all $w\in\left[\frac{w_*(\infty)+\overline{w}(\infty)}{2},\overline{w}(s)\right]$ and sufficiently large $s$, $\overline w(s), w_*(s)$ and $h(s,w)$ are bounded functions. Hence,  there exists $C>0$ such that 
     \[
     \begin{array}{llllll}
         0-\underline f(s)&=& F_2(w,h(\infty,w))-F_2(w,h(s,w))\\
         &=&\displaystyle \sum_{i=1}^n\left(\arctan(a_iw+h(\infty,w))-\arctan(a_iw+h(s,w))\right)\\
         &\leq & C(h(\infty,w)-h(s,w)).
     \end{array}
     \]
     Inequality in \eqref{equ-AsymBehav-hsw-Supsol} follows from the asymptotic behavior of $\underline f$. This finishes the proof of this lemma.
 \end{proof} 
 \begin{corollary}\label{Coro-SupsolODE} 
    Let  $v$ be as in \eqref{equ-def-GeneraSym}.
    If $V\in C^2([0,+\infty))$, 
     \[
     0<V'\leq 1,\quad V''\geq 0\quad\text{and}\quad 2d_n(s+1)V''\leq h(s,V')\quad\text{in }s\geq 0,
     \]
     Then $v$ is a supersolution to $P[v]=\Xi-f(x,t)$ in $\mathbb R^{n+1}_-$. 
 \end{corollary}
 \begin{proof}
     From the definition of generalized symmetric function as in \eqref{equ-def-GeneraSym}, by Lemmas \ref{lem-Supsol-EigenValEst} and \ref{lem-InverseFunc-Supsol}, 
     \[
     \begin{array}{llllll}
         P[v]&\leq &\displaystyle  -\widetilde\tau V'+\sum_{i=1}^n\arctan\left(a_iV'+\sum_{i=1}^nd_j^2x_j^2V''\right)\\
         &\leq & \displaystyle -\widetilde\tau V'+\sum_{i=1}^n\arctan\left(
             a_iV'+2d_n(s+1)V''
         \right)\\
         &\leq & F_2(V',h(s,V'))\\
         &= & \Xi+\underline f(s),
     \end{array}
     \]
     in $\mathbb R^{n+1}_-$. This finishes the proof of this lemma.
 \end{proof}
 
 Similar to the proof of Lemma \ref{lem-ODEInitial-Subsol-ConsRHS} (see also  Lemma 14 in \cite{Bao-Liu-Wang-DiricLagrangianEntire}), we have the following existence, uniqueness result on initial value problem of ordinary differential equation along with its monotonicity and asymptotic behavior at infinity. 
 
 \begin{lemma}\label{lem-ODEInitial-Supsol-ConsRHS}
    For any $w_0\in (w_*(0),\overline w(0))$, there exists a unique solution $W\in C^1([0,+\infty))$ to the initial value problem  
     \begin{equation}\label{equ-ODEInitial-Supsol-ConsRHS}
         \left\{
             \begin{array}{llll}
                 \displaystyle \dfrac{\mathrm dw}{\mathrm ds}=\dfrac{h(s,w)}{2d_n(s+1)}, & \text{in }s\geq 0,\\
                 w(0)=w_0.
             \end{array}
         \right.
     \end{equation}
     Furthermore, $W'(s)>0$, $w_*(s)< W(s)<\overline w(s)$, and $W$ satisfies \eqref{equ-Asym-InitSols} at infinity.
 \end{lemma}
  
 Corresponding to Lemma \ref{lem-ConsTru-Subsol-Exterio-Trans}, we have the following existence of supersolutions. 
For any given $w_0\in(\underline{w}(0),w^*(0))$, let 
    \[
     \overline v(x,t):=\int_0^{s(x,t)}W(\tau)\mathrm d \tau+C,\quad\forall~(x,t)\in\mathbb R^{n+1}_-,
    \]
     where $W(\tau)$ is the solution to problem \eqref{equ-ODEInitial-Supsol-ConsRHS} with  initial value $W(0)=w_0$ constructed in Lemma \ref{lem-ODEInitial-Supsol-ConsRHS}, $C$ is a constant to be determined. Again, by the fact that $M(A,\Xi)>\frac{\beta}{2}>1$, we may choose suitable $C$ to match the asymptotic behavior at infinity, which leads to the following result. 
 \begin{lemma}\label{lem-ConsTru-Supsol-Exterio-Trans}
     There exists a constant $C$ such that the function $\overline v(x,t)$ defined above 
     satisfies 
     \[
     \left\{
         \begin{array}{llllll}
             P[\overline v]\leq\Xi-f(x,t), & \text{in }\mathbb R^{n+1}_-,\\
             \displaystyle \overline v(x,t)=\widetilde\tau t+\frac{1}{2}x'Dx+O(s^{1-\frac{\beta}{2}}(x,t)), & \text{as }-t+|x|^2\rightarrow\infty.
         \end{array}
     \right.
     \]
 \end{lemma}
 The proof of Lemmas \ref{lem-ODEInitial-Supsol-ConsRHS} and \ref{lem-ConsTru-Supsol-Exterio-Trans} are omitted here, since they are almost identical to the proof of Lemmas \ref{lem-ODEInitial-Subsol-ConsRHS} and \ref{lem-ConsTru-Subsol-Exterio-Trans} respectively.

 \subsection{Proof of Theorem \ref{thm-Main3-EntireDiri}}\label{seclabel-Subsec-ProofExistence}
 
 By the results in Lemmas \ref{lem-ConsTru-Subsol-Exterio-Trans} and \ref{lem-ConsTru-Supsol-Exterio-Trans}, we have  viscosity subsolution $\underline v$ and viscosity supersolution $\overline v$ satisfying 
 \[
 P[\underline v](x,t)\geq \Xi-f(x,t)\quad\text{and}\quad 
 P[\overline v](x,t)\leq \Xi-f(x,t)\quad\text{in }\mathbb R^{n+1}_-,
 \]
 with  asymptotic behavior 
 \begin{equation}\label{equ-temp-SameAsymBehav}
    \underline v(x,t),~  \overline v(x,t)=\widetilde\tau t+\frac{1}{2}x'Dx+(s^{1-\frac{\beta}{2}}(x,t)),\quad \text{as }-t+|x|^2\rightarrow\infty.
 \end{equation}
Applying comparison principle (see for instance \cite[Chapter 14]{Book-Lieberman-SecondParboDE}), we have 
\[
\underline v(x,t)\leq \overline v(x,t)\quad\text{in }\mathbb R^{n+1}_-.
\]
Let 
\[
\mathcal S_{\overline v}:=\{v~:~v\text{ is a weak viscosity subsolution to }P[v]=\Xi-f\text{ and }v\leq \overline v\text{ in }\mathbb R^{n+1}_-\}.
\]
From the results above, $\underline v\in\mathcal S_{\overline v}$. 
Therefore, by Lemma \ref{lem-PerronMethod-WeakSol}, 
\[
v(x,t):=\sup\{\widetilde v(x,t)~:~\widetilde{v}\in\mathcal S_{\overline v}\}
\]
is a weak viscosity solution to $P[v]=\Xi-f$ in $\mathbb R^{n+1}_-$. Furthermore, since $\underline v$ and $\overline v$ satisfy \eqref{equ-temp-SameAsymBehav}, it follows from definition that $v$ becomes a weak viscosity solution to \eqref{equ-ExteriorDiri-Translated}. 

From the definition of USC envelop and LSC envelop, we have 
\[
v_*(x,t)\leq v(x,t)\leq v^*(x,t)\quad\text{in }\mathbb R^{n+1}_-.
\]
On the other hand, $v_*$ and $v^*$ are viscosity supersolution and viscosity subsolution to $P[v]=\Xi-f(x,t)$ in $\mathbb R^{n+1}_-$ respectively, with asymptotic behavior \eqref{equ-temp-SameAsymBehav}. By comparison principle,  
\[
v^*(x,t)\leq v_*(x,t)\quad\text{in }\mathbb R^{n+1}_-.
\]
Therefore, $v_*=v=v^*$, which implies that $v$ is continuous in $\mathbb R^{n+1}_-$. Since $v=v_*$ is  a viscosity supersolution and $v=v^*$ is a viscosity subsolution to \eqref{equ-ExteriorDiri-Translated}, by definition, $v$ is a viscosity solution to \eqref{equ-ExteriorDiri-Translated}. This finishes the proof of Theorem \ref{thm-Main3-EntireDiri}.

When $f\in C^0(\mathbb R^{n+1}_-)$ has compact support, for any $\zeta>0$, we only need to take $\beta>\zeta+2$ in the proof above. Following the same procedure, we obtain a viscosity solution to \eqref{equ-Prob-EntireDiri-InThm} with enhanced asymptotic behavior given in \eqref{equ-enhanced-AsymBehav}.

\section{Proof of Theorems \ref{thm-Main1-AsymBehav} and \ref{thm-main5-GeneralThm}}\label{seclabel-Theorem-AsymBehav}

In this section, we apply the Liouville type rigidity in $\mathbb R^n\times(-\infty,-T]$, which transforms the study of  asymptotic behavior into an initial value problem. Therefore, the comparison principle without growth condition (Theorem \ref{thm-Main2-Uniqueness-Extended}) and the existence results in Theorem \ref{thm-Main3-EntireDiri} imply the asymptotic behavior at infinity of ancient solutions. 

We will prove Theorem \ref{thm-main5-GeneralThm} i.e.,  the polynomial convergence rate for viscosity solutions,  in subsection \ref{seclabel-PolyNomConv-Subsec}. Then, refine it into exponential convergence rate for classical solutions with bounded Hessian matrix i.e.  Theorem \ref{thm-Main1-AsymBehav},  in subsection \ref{seclabel-ExpoNenConv-Subsec}.  

\subsection{The convergence rate for viscosity solutions}\label{seclabel-PolyNomConv-Subsec}
 
From the choice of $T$ in \eqref{equ-def-T}, $f(x,t)\equiv 0$ in $\mathbb R^{n}\times(-\infty,-T]$. Therefore, under any condition such that Liouville type rigidity holds, there exist $A\in\mathrm{Sym}(n), b\in\mathbb R^n,  c \in\mathbb R$ and $\tau:=\sum_{i=1}^n\arctan\lambda_i(A)$ such that 
\begin{equation}\label{equ-temp-rigidityResult}
u(x,t)=\tau t+\frac{1}{2}x'Ax+b\cdot x+c,\quad\forall~(x,t)\in\mathbb R^n\times(-\infty,-T].
\end{equation}

Consequently, under the conditions in Theorems \ref{thm-Main1-AsymBehav} or \ref{thm-main5-GeneralThm}, $u$ becomes a classical (or viscosity, respectively) solution to an initial value type problem 
\begin{equation}\label{equ-InitialValProb-InProof2}
\left\{
    \begin{array}{lllll}
        \displaystyle u_t=\sum_{i=1}^n\arctan\lambda_i(D^2u)+f(x,t), & \text{in }\mathbb R^n\times(-T,0],\\
        \displaystyle u(x,-T)=\frac{1}{2}x'Ax+b\cdot x+\widetilde c, & \text{in }\mathbb R^n,
    \end{array}
\right.
\end{equation}
with $\widetilde c:=c-\tau T$. 
% In other words, $v$ is a classical (or viscosity, respectively) solution to 
% \begin{equation}\label{equ-InitialValProb-InProof}
% \left\{
%     \begin{array}{lllll}
%         \displaystyle v_t=\sum_{i=1}^n\arctan\lambda_i(D^2v)+f(x,t-T), & \text{in }\mathbb R^n\times(0,T],\\
%         \displaystyle v(x,0)=\frac{1}{2}x'Ax+b\cdot x+\widetilde c , & \text{in }\mathbb R^n.
%     \end{array}
% \right.
% \end{equation}

For $A,b,c,\tau,f$ as above, let $\Upsilon\in C^0(\mathbb R^{n+1}_-)$ be the unique viscosity solution to 
\[
\left\{
    \begin{array}{lllll}
        \displaystyle\Upsilon_t=\sum_{i=1}^n\arctan\lambda_i(D^2\Upsilon)+f(x,t), & \text{in }\mathbb R^{n+1}_-,\\
        \displaystyle \Upsilon(x,t)=\tau t+\frac{1}{2}x'Ax+b\cdot x+c+O(\mathcal R^{-\zeta}(x,t)), & \text{as }\mathcal R(x,t)\rightarrow\infty,
    \end{array}
\right.
\]
where $\zeta>0$ can be arbitrarily large. 
Based on the fact that $f$ having compact support, the existence of $\Upsilon$ is guaranteed by Theorem \ref{thm-Main3-EntireDiri}.
Applying comparison principle (the classical one, see for instance Theorem 14.1 in \cite{Book-Lieberman-SecondParboDE}) with respect to $\Upsilon$ on $\mathbb R^n\times(-\infty,-T]$, it proves that
\[
\Upsilon(x,t)=\tau t+\frac12x'Ax+b\cdot x+ c,\quad\text{in }\mathbb R^n\times (-\infty,-T],
\]
and there exists $C>0$ such that 
\[
\left|\Upsilon(x,t)-\left(\tau t+\frac{1}{2}x'Ax+b\cdot x+c\right)\right|\leq C|x|^{-\zeta},\quad\text{in }\mathbb R^n\times[-T,0].
\]
Therefore, $\Upsilon$ is also a viscosity solution to initial value problem \eqref{equ-InitialValProb-InProof2}.
After a translation in $t$-variable, we may apply   Theorem \ref{thm-Main2-Uniqueness-Extended} and obtain 
\[
u(x,t)=\Upsilon(x,t),\quad\forall~(x,t)\in\mathbb R^{n}\times[-T,0]. 
\]
Therefore, 
\[
\left|u(x,t)-\left(\tau t+\frac{1}{2}x'Ax+b\cdot x+c\right)\right|\leq C|x|^{-\zeta},\quad\forall~(x,t)\in\mathbb R^{n}\times[-T,0].
\]
By the arbitrariness of $\zeta>0$,
this finishes the proof of Theorem \ref{thm-main5-GeneralThm}, as  in \eqref{equ-AsymBehav-Rough-C0Est-E}.

Using a similar argument, we provide the proof of Corollary \ref{Coro-InitialValProb-AsymBehav}. 

\begin{proof}[Proof of Corollary \ref{Coro-InitialValProb-AsymBehav}]
    Let $v(x,t):=u(x,t+\overline T)$ in $\mathbb R^n\times[-\overline T,0]$, which is a viscosity solution to 
    \begin{equation}\label{equ-temp-DirichletCoro}
    \left\{
        \begin{array}{llll}
            \displaystyle v_t=\sum_{i=1}^n\arctan\lambda_i(D^2v)+\widetilde f(x,t), & \text{in }\mathbb R^n\times[-\overline T,0],\\
            \displaystyle v(x,-\overline T)=\frac{1}{2}x'Ax+b\cdot x+c, & \text{in }\mathbb R^n,
        \end{array}
    \right.
    \end{equation}
    where
    \[
    \widetilde f(x,t):=\left\{
        \begin{array}{lllll}
            f(x,t+\overline T), & (x,t)\in\mathbb R^n\times[-\overline T,0],\\
            0, & (x,t)\in\mathbb R^n\times(-\infty,-\overline T),
        \end{array}
    \right.
    \]
    By definition, $\widetilde f\in C^0(\mathbb R^{n+1}_-)$ satisfies condition \eqref{equ-cond-RHS-ConverSpeed} for the same $\beta>2$ given in \eqref{equ-cond-InitialValProb-Coro}. 
    By Theorem \ref{thm-Main3-EntireDiri}, 
    let $\Upsilon\in C^0(\mathbb R^{n+1}_-)$ be the viscosity solution to 
    \[
    \left\{
        \begin{array}{lllll}
            \displaystyle \Upsilon_t=\sum_{i=1}^n\arctan\lambda_i(D^2\Upsilon)+\widetilde f(x,t), & \text{in }\mathbb R^n\times(-\infty,0],\\
            \displaystyle \Upsilon(x,t)=-\tau t+\frac{1}{2}x'Ax+b\cdot x+\widetilde c+O(\mathcal R^{2-\beta}(x,t)), & \text{as }\mathcal R(x,t)\rightarrow\infty,
        \end{array}
    \right.
    \]
    where $\widetilde c:=c-\tau \overline T.$
    Applying comparison principle on $\mathbb R^n\times(-\infty,-\overline T]$, we find that 
    \[
    \Upsilon(x,t)=\frac{1}{2}x'Ax+b\cdot x+ c,\quad\forall~(x,t)\in \mathbb R^n\times(-\infty,-\overline T].
    \]
    Therefore, $\Upsilon$ is also a viscosity solution to \eqref{equ-temp-DirichletCoro}. Using Theorem \ref{thm-Main2-Uniqueness-Extended},   we have 
    \[
    v\equiv\Upsilon,\quad\text{i.e.,}\quad u(x,t)=\Upsilon(x,t-\overline T),\quad\forall~(x,t)\in\mathbb R^n\times[0,\overline T].
    \]
    Thus, the asymptotic behavior of $v(x,t)$ at infinity follows immediately from the asymptotic behavior of $\Upsilon$. This finishes the proof of Corollary \ref{Coro-InitialValProb-AsymBehav}. 
\end{proof}

\subsection{Exponential convergence rate for  solutions with interior estimates}\label{seclabel-ExpoNenConv-Subsec}

In this subsection, we improve the convergence speed into exponential rate, using the linearization method and  the fundamental solution to linear parabolic equations, see for instance \cite{Book-Friedman-PDE-ParaboForm}.
The existence of fundamental solution require that the linearized equation have uniformly parabolic coefficients that are bounded in some H\"older continuous space as in $C^{\alpha,\frac{\alpha}{2}}(\mathbb R^{n+1}_-), 0<\alpha<1$. Therefore, the proof only work for classical solutions, instead of viscosity solutions.   

In summary, once the following theorem is proven, Theorem \ref{thm-Main1-AsymBehav}  will follow immediately. The exponential decay is achieved through an iterative linearization scheme, leveraging the uniform 
parabolicity induced by bounded Hessian matrix. 
\begin{theorem}\label{thm-Mainthm-Asym-InSummary}
    Let $n,f,u,E$ be  as in Theorem \ref{thm-Main1-AsymBehav}. If $u\in C^{2,1}(\mathbb R^{n+1}_-)$ is a classical solution to \eqref{equ-LagFlow-Extended} in $\mathbb R^{n+1}_-$ and satisfies condition \eqref{equ-cond-HessianBdd}. Then the polynomial convergence rate \eqref{equ-AsymBehav-Rough-C0Est-E} can be improved into \eqref{equ-Result-AsymBehav-u}. 
\end{theorem}

As in the proof of Theorem  \ref{thm-Main3-EntireDiri}, we may assume without loss of generality that $A$ is diagonal, $b=0$ and $c=0$.

\begin{lemma}\label{lem-AsymBehav-HessianMatrix}
    Let $n,f,u,E$ be as in Theorem \ref{thm-Mainthm-Asym-InSummary}. 
    Suppose $E$ satisfies polynomial convergence rate \eqref{equ-AsymBehav-Rough-C0Est-E} for some $\zeta>0$ and $C>0$. Then $||D^2E||_{C^{\alpha,\frac\alpha2}(\mathbb R^{n+1}_-)}$ is bounded for any $0<\alpha<1$, and 
    \[
    |D^2E(x,t)|+|E_t(x,t)|\leq C\mathcal R^{-\zeta-2}(x,t),\quad\forall~(x,t)\in\mathbb R^{n+1}_-,
    \]
    for some positive constant $C>0$. 
\end{lemma}
\begin{proof}
    For any sufficiently large $R>2$ and any point $(y,\tau)\in\mathbb R^{n+1}_-$ with $\mathcal R(y,\tau)=R$, set
    \[
    D:=\left\{(z,\zeta)~:~\left(|\zeta|+\frac{1}{2}|z|^2\right)^{\frac{1}{2}}\leq\frac{3}{2}\quad\text{and}\quad 
    \zeta\leq -\frac{16\tau}{R^2}\right\},
    \]
    \[
    u_R(z,\zeta):=\left(\frac{4}{R}\right)^2
    u\left(y+\frac{R}{4}z,\tau+\frac{R^2}{16}\zeta\right),\quad(z,\zeta)\in D, 
    \]
    and   
    \[
    \begin{array}{lllll}
    E_R(z,\zeta)&:=&\displaystyle \left(\frac{4}{R}\right)^2
    E\left(y+\frac{R}{4}z,\tau+\frac{R^2}{16}\zeta\right)\\ 
    &=&\displaystyle u_R(z,\zeta)-\tau\left(\zeta+\frac{16}{R^2}\tau\right)-8\left(\frac{y}{R}+\frac{z}{4}\right)'A\left(\frac{y}{R}+\frac{z}{4}\right).
    \end{array}
    \]
    By the convergence rate \eqref{equ-AsymBehav-Rough-C0Est-E}, there exists $C>0$ such that for sufficiently large $R$,
    \[ 
    |u_R(z,\zeta)|\leq C\quad \text{and}\quad |E_R(z,\zeta)|\leq CR^{-\zeta-2},\quad\forall~(z,\zeta)\in D.
    \]
    Since $u$ is a classical solution having bounded Hessian matrix and $f$ has compact support, we find that for sufficiently large $R$, $u_R\in C^{2,1}(D)$ has bounded Hessian matrix and satisfies 
    \[
    \begin{array}{lllll}
        (u_R)_\zeta(z,\zeta)&=& \displaystyle u_t\left(y+\frac{R}{4}z,\tau+\frac{R^2}{16}\zeta\right)\\
        &=&\displaystyle \sum_{i=1}^n\arctan\lambda_i\left(D^2u
    \left(y+\frac{R}{4}z,\tau+\frac{R^2}{16}\zeta\right)
    \right)\\
    &=&\displaystyle \sum_{i=1}^n\arctan\lambda_i(D^2u_R(z,\zeta)),\quad\forall~(z,\zeta)\in D.
    \end{array}
    \]
    By taking partial derivatives with respect to $u_R$ and applying Krylov-Safonov H\"older inequality (see for instance Page 133 in \cite{Book-Krylov-Nonlinear-EllParabo-SecondOrder} or Theorem 4.1 in \cite{Krylov-Safonov-HarnackInequ-ParaboEqu}), for any $k\in\mathbb N$, there exists $C>0$ (which may vary from line to line) such that 
    \[
    ||u_R||_{C^{k,k/2}(\widetilde D)}\leq C,\quad C^{-1}I\leq D^2u_R\leq CI\quad\text{in }\widetilde D,
    \]
    where $\widetilde D$ denote a compact subset of $D$ that contain $(0,0)$ as an interior point. Consequently, for any $k\in\mathbb N$, there exists $C>0$ such that  for all sufficiently large $R$, 
    \[
        ||E_R||_{C^{k,k/2}(\widetilde D)}\leq C,\quad 
    C^{-1}I\leq (A+D^2E_R)\leq CI\quad\text{in }\widetilde D.
    \]
    By the Newton-Leibnitz formula, $E_R$ satisfies a parabolic equation 
    \[
    (E_R)_\zeta(z,\zeta)-\sum_{i,j=1}^na^{ij}_R(z,\zeta)D_{ij}E_R=0\quad\text{in }D,
    \]
    where $(a^{ij}_R(z,\zeta))_{n\times n}$ is a bounded, strictly positive matrix with elements
    \[
    a^{ij}_R(z,\zeta)=\int_0^1D_{M_{ij}}F(
    A+\theta D^2E_R(z,\zeta))\mathrm d \theta,
    \]
    and $D_{M_{ij}}F$ denotes the partial derivative of $F(M):=\sum_{i=1}^n\arctan\lambda_i(M)$ with respect to $M_{ij}$ variable. 
    By the interior Schauder estimates (see for instance Theorem 5 in \cite[Chapter 3]{Book-Friedman-PDE-ParaboForm}), 
    \[
    |D_z^iD_{\zeta}^jE_R(0,0)|\leq C||E_R||_{C^0(\widetilde D)}\leq CR^{-\zeta-2},\quad\forall~i+2j=k\geq 2.
    \]
    It follows that for all $(y,\tau)\in\mathbb R^{n+1}_-$, 
    \[
    |D_x^iD_t^jE(y,\tau)|\leq C\mathcal R^{-\zeta-k}(y,\tau),\quad\forall~i+2j=k\geq 2.
    \]
    This finishes the proof of the desired result. 
\end{proof}

To proceed, we introduce some calculus fact, which is closely related to the fundamental solutions of second order linear parabolic equations. 
By a change of variable, we have the following anisotropic version of Lemma 3 in \cite[Chapter 1]{Book-Friedman-PDE-ParaboForm}.

\begin{lemma}\label{lem-Calculus-Convolution}
    Let $\kappa_1,\kappa_2,\cdots,\kappa_n$ be a sequence of positive constants and 
    \[
    \kappa(x):=\sum_{i=1}^n\kappa_i|x_i|^2,\quad\forall~x=(x_1,x_2,\cdots,x_n)\in\mathbb R^n.
    \]
    For any $\alpha,\beta\in(-\infty,\frac{n}{2}+1)$,
    \[
    \begin{array}{lllll}
        &\displaystyle \int_{\sigma}^t\int_{\mathbb R^n}\frac{1}{(t-\tau)^\alpha}\exp\left(-\frac{\kappa(x-\xi)}{4(t-\tau)}\right)
        \cdot\frac{1}{(\tau-\sigma)^\beta}\exp\left(
            -\frac{\kappa(\xi-y)}{4(\tau-\sigma)}
        \right)\mathrm d \xi\mathrm d \tau\\
        =&\displaystyle \frac{(4\pi)^{\frac{n}{2}}}{(\prod_{i=1}^n\kappa_i)^{\frac{1}{2}}}B\left(\frac{n}{2}-\alpha+1,\frac{n}{2}-\beta+1\right) (t-\sigma)^{\frac{n}{2}+1-\alpha-\beta}\cdot \exp\left(-\frac{\kappa(x-y)}{4(t-\sigma)}\right),
    \end{array}
    \]
    where $B(\cdot,\cdot)$ denote the Beta function 
    \[
    B(a,b):=\int_0^1t^{a-1}(1-t)^{b-1}\mathrm dt.
    \]
  \end{lemma}
\begin{proof}
    If $\kappa_i\equiv h>0$ for all $i=1,2,\cdots,n$, it is exactly the identity in \cite{Book-Friedman-PDE-ParaboForm}. Change of variable by setting 
    \[
    \widetilde x:=\left(\kappa_1^{\frac12}x_1,\cdots,\kappa_n^{\frac12}x_n\right),\quad 
    \widetilde\xi:=\left(\kappa_1^{\frac12}\xi_1,\cdots,\kappa_n^{\frac12}\xi_n\right)\quad\text{and}\quad 
    \widetilde y:=\left(\kappa_1^{\frac12}y_1,\cdots,\kappa_n^{\frac12}y_n\right).
    \]
    By a direct computation, 
    \[
    \kappa(x-\xi)=\sum_{i=1}^n\kappa_i|x_i-\xi_i|^2
    =\sum_{i=1}^n|\widetilde x_i-\widetilde\xi_i|^2=|\widetilde x-\widetilde\xi|^2.
    \]
    Similarly, $\kappa(\xi-y)=|\widetilde\xi-\widetilde y|^2$. Hence, by the identities above and Lemma 3 in \cite[Chapter 1]{Book-Friedman-PDE-ParaboForm},
    \[
        \begin{array}{lllll}
            &\displaystyle \int_{\sigma}^t\int_{\mathbb R^n}\frac{1}{(t-\tau)^\alpha}\exp\left(-\frac{\kappa(x-\xi)}{4(t-\tau)}\right)
            \cdot\frac{1}{(\tau-\sigma)^\beta}\exp\left(
                -\frac{\kappa(\xi-y)}{4(\tau-\sigma)}
            \right)\mathrm d \xi\mathrm d \tau\\
            =&\displaystyle \frac{1}{(\prod_{i=1}^n\kappa_i)^{\frac12}}\int_{\sigma}^t\int_{\mathbb R^n}\frac{1}{(t-\tau)^\alpha}\exp\left(-\frac{|\widetilde x-\widetilde\xi|^2}{4(t-\tau)}\right)
            \cdot\frac{1}{(\tau-\sigma)^\beta}\exp\left(
                -\frac{|\widetilde\xi-\widetilde y|^2}{4(\tau-\sigma)}
            \right)\mathrm d \widetilde\xi\mathrm d \tau\\
            =&\displaystyle \frac{(4\pi)^{\frac{n}{2}}}{(\prod_{i=1}^n\kappa_i)^{\frac{1}{2}}}B\left(\frac{n}{2}-\alpha+1,\frac{n}{2}-\beta+1\right) (t-\sigma)^{\frac{n}{2}+1-\alpha-\beta}\cdot \exp\left(-\frac{|\widetilde x-\widetilde y|^2}{4(t-\sigma)}\right).
        \end{array}
    \] 
    This finishes the proof of the desired identity.
\end{proof}

Secondly, we establish the following estimate for  the integral of the product of the  fundamental solution and  a positive function with compact support. 

\begin{lemma}\label{lem-Calculus-Convolution-CompactCase}
    Let $\kappa$ be as in Lemma \ref{lem-Calculus-Convolution} and $f\in C^0(\mathbb R^n\times[-T,0])$ be a positive function supported in $\overline{B_{S_0}}\times[-T,0]$. Then there exist $S_1>2S_0$ and $C>0$ such that for all $|x|\geq S_1$ and $t\in[-T,0]$,
    \[
    \begin{array}{lll}
        &\displaystyle\int_{-T}^t\int_{\mathbb R^n}\frac{1}{(t-\tau)^{\frac{n}{2}}}\exp\left(-\frac{\kappa(x-\xi)}{4(t-\tau)}\right)\cdot f(\xi,\tau)\mathrm d\xi\mathrm d\tau\\
        \leq & C \displaystyle  \frac{1}{(t+T)^{\frac{n}{2}-2}}\exp\left(-\frac{\kappa(x)}{4(t+T)}+\sum_{i=1}^n\frac{2\kappa_iS_0|x_i|}{4(t+T)}\right).
    \end{array}
    \]
    Especially, for any $\epsilon>0$, there exist $S_1'>S_1$ and $C>0$ such  that for all $|x|\geq S_1'$ and $t\in[-T,0]$,
    \[
        \int_{-T}^t\int_{\mathbb R^n}\frac{1}{(t-\tau)^{\frac{n}{2}}}\exp\left(-\frac{\kappa(x-\xi)}{4(t-\tau)}\right)\cdot f(\xi,\tau)\mathrm d\xi\mathrm d\tau
    \leq C\exp\left(-\frac{\kappa(x)-\epsilon|x|^2}{4(t+T)}\right).
    \]
\end{lemma}

\begin{proof}
    By the triangle inequality and the restriction $\xi\in\overline{B_{S_0}}$, we have 
    \[
    |x_i-\xi_i|^2=|x_i|^2-2x_i\cdot\xi_i+|\xi_i|^2\geq |x_i|^2-2S_0|x_i|.
    \]
    Consequently, for all $(x,t)\in\mathbb R^n\times[-T,0]$, 
    \[
    \begin{array}{lllll}
        & \displaystyle \int_{-T}^t\int_{\mathbb R^n}\frac{1}{(t-\tau)^{\frac{n}{2}}}\exp\left(-\frac{\kappa(x-\xi)}{4(t-\tau)}\right)\cdot f(\xi,\tau)\mathrm d\xi\mathrm d\tau\\
    \leq & \displaystyle \max_{\overline{B_{S_0}}\times[-T,0]}|f|\cdot 
    \int_{-T}^t\int_{B_{S_0}}\frac{1}{(t-\tau)^{\frac{n}{2}}}
    \exp\left(-\sum_{i=1}^n\frac{\kappa_i|x_i-\xi_i|^2}{4(t-\tau)}\right)\mathrm d \xi\mathrm d \tau\\
    \leq &\displaystyle \max_{\overline{B_{S_0}}\times[-T,0]}|f|\cdot 
    \int_{-T}^t\frac{1}{(t-\tau)^{\frac{n}{2}}} \int_{B_{S_0}}
    \exp\left(-\sum_{i=1}^n\frac{\kappa_i|x_i|^2-2\kappa_iS_0|x_i|}{4(t-\tau)}\right)\mathrm d \xi\mathrm d \tau\\
    \leq &\displaystyle  |B_{S_0}| \max_{\overline{B_{S_0}}\times[-T,0]}|f|\cdot 
    \int_{-T}^t\frac{1}{(t-\tau)^{\frac{n}{2}}}  
    \exp\left(-\sum_{i=1}^n\frac{\kappa_i|x_i|^2-2\kappa_iS_0|x_i|}{4(t-\tau)}\right) \mathrm d \tau.
    \end{array}
    \]
    We may pick $S_1>2S_0\displaystyle \max_{i,j=1,2,\cdots,n}\frac{\kappa_i}{\kappa_j}+1$ such that  
    \[
    \sum_{i=1}^n\left(\kappa_i|x_i|^2-2\kappa_iS_0|x_i|\right)
    \geq |x|^2\min_{i=1,2,\cdots,n}\kappa_i-2S_0|x|\max_{i=1,2,\cdots,n}\kappa_i>1,\quad\forall~|x|\geq \widetilde S_1.
    \] 
    Integrating by parts finite times implies that 
    there exist $C>0$  such that for all $|x|\geq S_1, t\in[-T,0]$,
    \[
    \begin{array}{lllll}
     & \displaystyle 
    \int_{-T}^t\frac{1}{(t-\tau)^{\frac{n}{2}}}  
    \exp\left(-\sum_{i=1}^n\frac{\kappa_i|x_i|^2-2\kappa_iS_0|x_i|}{4(t-\tau)}\right) \mathrm d \tau\\
    = &\displaystyle  \int^{t+T}_0\frac{1}{s^{\frac{n}{2}}}\exp\left(-\sum_{i=1}^n\frac{\kappa_i|x_i|^2-2\kappa_iS_0|x_i|}{4s}\right)\mathrm ds\\
    \leq &\displaystyle  C \frac{1}{(t+T)^{\frac{n}{2}-2}\cdot(\sum_{i=1}^n\kappa_i|x_i|^2-2\kappa_iS_0|x_i|)}\exp\left(-\sum_{i=1}^n\frac{\kappa_i|x_i|^2-2\kappa_iS_0|x_i|}{4(t+T)}\right).
    \end{array}
    \]
    By taking $S_1$ even larger such that 
    \[
    \frac{1}{(\sum_{i=1}^n\kappa_i|x_i|^2-2\kappa_iS_0|x_i|)}
    \leq \frac{2}{\sum_{i=1}^n\kappa_i|x_i|^2}=\frac{2}{\kappa(x)},\quad\forall~|x|\geq S_1.
    \]
    Combining the estimates above finishes the proof of the  first inequality. 

    To proceed, for any $\epsilon>0$, we may pick  $S_1' \geq \max\{S_1,\frac{4S_0}{\epsilon}\displaystyle \max_{i=1,2,\cdots,n}\kappa_i\}$ such that 
    \[
    \sum_{i=1}^n\left(\frac{\epsilon}{2}|x_i|^2-2\kappa_iS_0|x_i|\right)\geq 
    \frac\epsilon2|x|^2-2S_0|x|\max_{i=1,2,\cdots,n}\kappa_i> 0,\quad\forall~|x|>S_1'.
    \]
    Therefore, continuing the estimates above, there exist $C>0$ such that  for all $|x|\geq S_1', t\in[-T,0]$,
    \[
        \begin{array}{lllll}
            &\displaystyle \int_{-T}^t\int_{\mathbb R^n}\frac{1}{(t-\tau)^{\frac{n}{2}}}\exp\left(-\frac{\kappa(x-\xi)}{4(t-\tau)}\right)\cdot f(\xi,\tau)\mathrm d\xi\mathrm d\tau\\
        % \leq &\displaystyle C \frac{1}{(t+T)^{\frac{n}{2}-2}}\exp\left(-\sum_{i=1}^n\frac{\kappa_i|x_i|^2-2\kappa_iS_0|x_i|}{4(t+T)}\right)\\
        \leq &\displaystyle  C \frac{1}{(t+T)^{\frac{n}{2}-2}}\exp\left(-\sum_{i=1}^n\frac{(\kappa_i-\frac\epsilon2)|x_i|^2}{4(t+T)}\right)\\
        \leq  & \displaystyle  C \exp\left(- \frac{\kappa(x)-\epsilon|x|^2}{4(t+T)}\right).
        \end{array}
    \]
    In the last inequality of the estimate above, we used the fact that for all $|x|\geq S_1'$ and $t\in[-T,0]$,
    \[
    \frac{1}{(t+T)^{\frac{n}{2}-2}}\exp\left(-\frac{\epsilon|x|^2}{8(t+T)}\right)\leq 
    \frac{1}{(t+T)^{\frac{n}{2}-2}}\exp\left(-\frac{\epsilon(S_1')^2}{8(t+T)}\right)\leq C,
    \]
    for some positive constant $C$.
    This finishes the proof of the second inequality, and concludes the proof of this lemma.
\end{proof}

Now, we prove an intermediate convergence rate between \eqref{equ-AsymBehav-Rough-C0Est-E} and \eqref{equ-Result-AsymBehav-u}. 
\begin{lemma}\label{lem-IntermiConvSpeed-ED2E}
  Let $n,f,u,E$ be as in Theorem \ref{thm-Mainthm-Asym-InSummary}. Then there exists $\upsilon_0>0$ such that for every  $\upsilon_0^*\in (0,\upsilon_0)$,
  \begin{equation}\label{equ-ConverSpeed-ExponeRough}
  |E(x,t)|+|D^2E(x,t)| 
           \leq C\left(\exp\left(-\frac{\upsilon_0^*|x|^2}{4(t+T)}\right)\right), \quad \forall~|x|\geq S_0,~t\in(-T,0], 
    \end{equation}
    for some $C>0$ and $S_0>S.$
\end{lemma}
\begin{proof}
  By the Newton-Leibnitz formula, $E(x,t)$ satisfies  
  \[
  \left\{
  \begin{array}{llllll}
  \displaystyle L[E](x,t):=
  \sum_{i,j=1}^na_{ij}(x,t)D_{ij}E(x,t)-E_t(x,t)=-f(x,t), & \text{in }\mathbb R^{n}\times[-T,0],\\
  E(x,-T)=0, & \text{in }\mathbb R^n,\\
  E(x,t)=O(|x|^{-\zeta}), & \text{as }|x|\rightarrow\infty,~\forall~t\in(-T,0],
  \end{array}
  \right.
  \]
  where the coefficients 
  \[
  a_{ij}(x,t):=\int_0^1D_{M_{ij}}F(A+\theta D^2E(x,t))\mathrm d \theta,
  \]
  have bounded $C^{\alpha,\frac{\alpha}{2}}(\mathbb R^{n+1}_-)$ norm (see Lemma \ref{lem-AsymBehav-HessianMatrix}), with a limit
  \[
  a_{ij}(\infty):=\lim_{-t+|x|^2\rightarrow\infty}a_{ij}(x,t)=D_{M_{ij}}F(A).
  \]
By Theorems 10-12 in \cite[Chapter 1]{Book-Friedman-PDE-ParaboForm}, we have a fundamental solution $\Gamma(x,t;y,\tau)$ associated to $L$, with the following estimates. For any $\upsilon_0^*<\upsilon_0$, there exists $C>0$ such that 
  \[
  |\Gamma(x,t;y,\tau)|\leq C\dfrac{1}{(t-\tau)^{\frac{n}{2}}}\exp\left(-\frac{\frac{\upsilon_0^*+\upsilon_0}{2}|x-y|^2}{4(t-\tau)}\right),\quad\forall~x,y\in\mathbb R^n,~-T\leq \tau<t\leq 0,
  \]
  where $\upsilon_0$ is the positive constant such that 
  \[
  \upsilon_0|\xi|^2\leq\sum_{i,j=1}^na^{ij}(x,t)\xi_i\xi_j\leq C|\xi|^2,\quad\forall~(x,t)\in\mathbb R^{n}\times[-T,0],~\xi\in\mathbb R^n,
  \]
  and $a^{ij}(x,t)$ denote the $i,j$-position of the inverse matrix of $(a_{ij}(x,t))_{n\times n}$. Therefore, by the initial value and the asymptotic behavior at infinity, comparison principle implies that 
  \[
  E(x,t)=-\int_{-T}^t\int_{\mathbb R^n}\Gamma(x,t;y,\tau)f(y,\tau)\mathrm dy\mathrm d\tau,\quad \forall~(x,t)\in\mathbb R^n\times[-T,0].
  \]
  Consequently, by Lemma \ref{lem-Calculus-Convolution-CompactCase}, there exists $S_0>S$ and $C>0$ such that for all $|x|\geq S_0$ and $t\in[-T,0]$, 
  \[
  \begin{array}{llllll}
   | E(x,t) | &\leq & \displaystyle C\int_{-T}^t\int_{\mathbb R^n}\frac{1}{(t-\tau)^{\frac{n}{2}}}\exp\left(-\frac{\frac{\upsilon_0^*+\upsilon_0}{2}|x-y|^2}{4(t-\tau)}\right)|f(y,\tau)|\mathrm dy\mathrm d\tau\\
   &\leq & \displaystyle 
   C \exp\left(-\frac{\frac{3\upsilon_0^*+\upsilon_0}{4}|x|^2}{4(t+T)}\right).
  \end{array}
  \] 
  This finishes the proof of the estimate on  $E$ at infinity. 
  
  The convergence rate of Hessian matrix $D^2E$ follows from a similar argument as in Lemma \ref{lem-AsymBehav-HessianMatrix}. More explicitly, for $E_R(z,\zeta)$ and $u_R(z,\zeta)$ defined as in Lemma \ref{lem-AsymBehav-HessianMatrix}, we restrict the domain $(z,\zeta)$ into 
  \[
  D':=\left\{
    (z,\zeta)~:~\left(|\zeta|+\frac{1}{2}|z|^2\right)^{\frac{1}{2}}\leq \epsilon\quad\text{and}\quad \zeta\leq -\frac{16\tau}{R^2} 
  \right\},
  \]
  for some $\epsilon>0$ to be determined and $\tau\in[-T,0]$. Then for large $R$, the estimate on $E$ implies that 
  \[
  \begin{array}{llll}
  |E_R(z,\zeta)|&\leq &\displaystyle  CR^{-2}\exp\left(-\frac{\frac{7\upsilon_0^*+\upsilon_0}{8}(1-\frac{\epsilon \sqrt 2}{4})(R^2-T)}{4(t+T)}\right)\\
  &\leq & \displaystyle CR^{-2}\exp\left(-\frac{\frac{15\upsilon_0^*+\upsilon_0}{16}(1-\frac{\epsilon\sqrt 2}{4})R^2}{4(t+T)}\right)\\
  &\leq & \displaystyle CR^{-2}\exp\left(-\frac{\upsilon_0^* R^2}{4(t+T)}\right),
  \end{array}
  \]
  by fixing $\epsilon>0$ sufficiently small. Consequently, by the interior Schauder estimate, $E_R$ satisfies 
  \[
  |D^2_zE_R(0,0)|\leq C||E_R||_{C^0(\overline D')}\leq CR^{-2}\exp\left(
    -\frac{\upsilon_0^* R^2}{4(t+T)}
  \right).
  \]
  It follows that for all $(y,\tau)\in\mathbb R^{n+1}_-$ with $|y|\geq S_0$ sufficiently large and  $\tau\in[-T,0]$, 
  \[
  |D^2E(y,\tau)|=|D^2_zE_R(0,0)|\leq CR^{-2}(y,\tau)\exp\left(-\frac{\upsilon_0^* R^2(y,\tau)}{4(t+T)}\right)
  \leq C\exp\left(-\frac{\upsilon_0^* R^2(y,\tau)}{4(t+T)}\right).
  \]
  This finishes the proof of the desired estimates on $|E(x,t)|$ and $|D^2E(x,t)|$. 
\end{proof}

Using the exponential convergence rate established in Lemma \ref{lem-IntermiConvSpeed-ED2E}, we devise an iteration scheme to refine the rate. 
By the estimates on $|D^2E(x,t)|$, the coefficients of linearized equation of  $E$ satisfies 
\[
|a_{ij}(x,t)-a_{ij}(\infty)|\leq C|D^2E(x,t)|,\quad\forall~(x,t)\in\mathbb R^n\times[-T,0],
\]
with 
\[
(a_{ij}(\infty))_{n\times n}=(I+A^2)^{-1}=\mathrm{diag}\left(\frac{1}{1+a_1^2},\frac{1}{1+a_2^2},\cdots,\frac{1}{1+a_n^2}\right).
\]
Correspondingly, 
\[
(a^{ij}(\infty))_{n\times n}=(I+A^2)=\mathrm{diag}(1+a_1^2,1+a_2^2,\cdots,1+a_n^2).
\]

\begin{lemma}\label{lem-IterationScheme-Exponential}
    Let $n,f,u,E$ be as in Theorem \ref{thm-Mainthm-Asym-InSummary}. Suppose $E$ satisfies asymptotic behavior \eqref{equ-ConverSpeed-ExponeRough} for some $\upsilon_0^*, C, S_0>0$. 
    For any $\varepsilon>0$, let 
    \[
    h_i^*:=h_i-\varepsilon\leq h_i:=\min\{1+a_i^2,2\upsilon_0^*\}.
    \]
    Then there exist $S_1'>S_0$ and $C>0$ such that  
    \[
        |E(x,t)|+|D^2E(x,t)|  
               \leq   C\left(\exp\left(-\sum_{i=1}^n
               \frac{h_i^*|x_i|^2}{4(t+T)}
               \right)\right),\quad  \forall~|x|\geq S_1',~t\in(-T,0]. 
    \]
\end{lemma}
\begin{proof}
    From the asymptotic behavior of $D^2E$ at infinity and the uniform ellipticity of equation,  there exists $C>0$ such that 
    \[
    \left|
    \sum_{i,j=1}^n
    (a_{ij}(x,t)-a_{ij}(\infty))\cdot D_{ij}E(x,t)
    \right|
    \leq C\exp\left(-\frac{2\upsilon_0^*|x|^2}{4(t+T)}\right),\quad\forall~|x|\geq S_0,~t\in [-T,0].
    \]
    Therefore, we may rewrite the linearized equation satisfied by $E$ into 
    \[
    \sum_{i,j=1}^nD_{M_{ij}}F(A)D_{ij}E(x,t)-E_t=-g(x,t)\quad\text{in }\mathbb R^n\times[-T,0],
    \]
    with 
    \[
    -g(x,t):=-f(x,t)+
        \sum_{i,j=1}^n\left(a_{ij}(\infty)-a_{ij}(x,t)\right)D_{ij}E(x,t).
    \]
    Furthermore, we estimate $|g(x,t)|$ by 
    \[
    |g(x,t)|\leq C\exp\left(-\frac{2\upsilon_0^*|x|^2}{4(t+T)}\right)+g_0(x,t),\quad\forall~(x,t)\in\mathbb R^n\times [-T,0]
    \]
    for some smooth, non-negative function $g_0(x,t)$ having compact support in $\overline{B_{S_0}}\times[-T,0]$. 
    Then, by the fundamental solution to linear parabolic equation with constant coefficients (see for instance \cite[Page 4, Chapter 1]{Book-Friedman-PDE-ParaboForm}) and the comparison principle, 
    \[
    E(x,t)=\int_{-T}^t\int_{\mathbb R^n}\frac{1}{(4\pi)^{\frac{n}{2}}}
    \frac{1}{(t-\tau)^{\frac{n}{2}}}\exp\left(
        -\sum_{i=1}^n\frac{(1+a_i^2)|x_i-y_i|^2}{4(t-\tau)}
    \right)g(y,\tau)\mathrm dy\mathrm d\tau.
    \] 
    By Lemmas \ref{lem-Calculus-Convolution} and \ref{lem-Calculus-Convolution-CompactCase}, there exist $S_1'>S_0$ and positive constants $C>0$  such that 
    \[
    \begin{array}{llll}
        & |E(x,t)|\\
        \leq & \displaystyle C\left|
            \displaystyle \int_{-T}^t\int_{\mathbb R^n} 
            \frac{1}{(t-\tau)^{\frac{n}{2}}}\exp\left(
                -\sum_{i=1}^n\frac{(1+a_i^2)|x_i-y_i|^2}{4(t-\tau)}
            \right)\exp\left(-\frac{2\upsilon_0^*|y|^2}{4(\tau+T)}\right)\mathrm dy\mathrm d\tau
        \right|\\
    &\displaystyle \quad +\int_{-T}^t\int_{B_{S_0}}
    \frac{1}{(t-\tau)^{\frac{n}{2}}}\exp\left(
                -\sum_{i=1}^n\frac{(1+a_i^2)|x_i-y_i|^2}{4(t-\tau)}
            \right)\cdot g_0(y,\tau) \mathrm dy\mathrm d\tau\\
    \leq & \displaystyle  C\left|
        \int_{-T}^t\int_{\mathbb R^n} 
        \frac{1}{(t-\tau)^{\frac{n}{2}}}\exp\left(
            -\sum_{i=1}^n\frac{h_i|x_i-y_i|^2}{4(t-\tau)}
        \right)\exp\left(-\sum_{i=1}^n\frac{h_i|y_i|^2}{4(\tau+T)}\right)\mathrm dy\mathrm d\tau
    \right|\\
    &\displaystyle\quad 
    + C\frac{1}{(t+T)^{\frac{n}{2}-2}}\exp\left(-\sum_{i=1}^n\frac{(1+a_i^2-\frac\varepsilon2)|x|^2}{4(t+T)}\right)
    \\
    \leq &\displaystyle  C \exp\left(-\sum_{i=1}^n\frac{h_i^*|x|^2}{4(t+T)}\right),\quad\forall~|x|\geq S_1',~t\in[-T,0].
    \end{array}
    \]
    The estimate of $|D^2E(x,t)|$ follows similarly as in the proof of Lemma \ref{lem-IntermiConvSpeed-ED2E}.
    This finishes the proof of the desired estimates. 
\end{proof} 
  
\begin{proof}[Proof of Theorem \ref{thm-Mainthm-Asym-InSummary}]
    Combining Lemmas \ref{lem-AsymBehav-HessianMatrix}, \ref{lem-IntermiConvSpeed-ED2E}, we may apply Lemma \ref{lem-IterationScheme-Exponential} repeatedly finite times. More explicitly, for any   $\upsilon_0^*\in (0,\upsilon_0)$ such that $E$ satisfies the desired estimates in \eqref{equ-ConverSpeed-ExponeRough},   there exists $k\geq 1$ such that 
    \[
        2^{k}\upsilon_0^*\geq 1+a_i^2,\quad\forall~i=1,2,\cdots,n.
    \]
    Applying Lemma \ref{lem-IterationScheme-Exponential} $k$ times, and picking $h_i^*$ sufficiently close to $h_i$, we find  $\mathfrak h_i<1+a_i^2<2\mathfrak h_i$  such that  
    \[
    |E(x,t)|+|D^2E(x,t)| 
               \leq   C\left(\exp\left(-\sum_{i=1}^n
               \frac{\mathfrak h_i|x_i|^2}{4(t+T)}
               \right)\right),\quad \forall~|x|\geq S_k,~t\in(-T,0],
    \]
    for some positive constants $S_k>0$ and $C>0$.
    From the proof of Lemma \ref{lem-IterationScheme-Exponential}, we similarly obtain that there exist a non-negative, smooth function $g_k$ with compact support in $\overline{B_{S_k}}\times[-T,0]$, $S_{k+1}>S_k'>S_k$ and $C>0$ such that
    \[
    \begin{array}{lll}
       & |E(x,t)|\\
       \leq &\displaystyle C\int_{-T}^t\int_{\mathbb R^n} 
        \frac{1}{(t-\tau)^{\frac{n}{2}}}\exp\left(
            -\sum_{i=1}^n\frac{(1+a_i^2)|x_i-y_i|^2}{4(t-\tau)}
        \right)\exp\left(-\sum_{i=1}^n\frac{(1+a_i^2)|y_i|^2}{4(+T)}\right)\mathrm dy\mathrm d\tau\\
        &\displaystyle \quad +\max_{\overline{B_{S_k}}\times[-T,0]}|g_k(x,t)|\cdot \int_{B_{S_k+1}}\int_{-T}^t\frac{1}{(t-\tau)^{\frac{n}{2}}}
        \exp\left(-\sum_{i=1}^n\frac{(1+a_i^2)|x_i-y_i|^2}{4(t-\tau)}\right)\mathrm d\tau\mathrm dy\\
        \leq &\displaystyle
         C\exp\left(-\sum_{i=1}^n\frac{(1+a_i^2)|x_i|^2}{4(t+T)}\right)\\
         &\displaystyle \quad 
         +C\frac{1}{(t+T)^{\frac{n}{2}-2}\cdot(\sum_{i=1}^n(1+a_i^2)|x_i|^2)}\exp\left(-\sum_{i=1}^n\frac{(1+a_i^2)|x|^2-2(1+a_i^2)S_k'|x_i|}{4(t+T)}\right),
    \end{array}
    \]
    for sufficiently large $|x|$ and $t\in[-T,0]$,
    where we used Lemmas \ref{lem-Calculus-Convolution} and \ref{lem-Calculus-Convolution-CompactCase} in the last line of the estimate above. By a direct computation, there exists $C>0$ such that
    \[
    (t+T)^{\frac{n}{2}-2}(\sum_{i=1}^n(1+a_i^2)|x_i|)\exp\left(-\sum_{i=1}^n\frac{2(1+a_i^2)S_k'|x_i|}{4(t+T)}\right)\leq C,
    \]
     for sufficiently large $|x|$ and $t\in [-T,0]$. The estimates above  finish the proof of \eqref{equ-Result-AsymBehav-u}.
\end{proof}

\begin{proof}[Proof of Theorem \ref{thm-Main1-AsymBehav}]
    Under the conditions as in  \eqref{cond-MainThm-1}, the desired results follow immediately from Theorems \ref{thm-main5-GeneralThm} and  \ref{thm-Mainthm-Asym-InSummary}. 

    Under the conditions as in \eqref{cond-MainThm-2}, we apply the same interior estimates from \cite{Bhattacharya-Warren-Weser-Liouville-LagranFlow-CPDE} to prove $D^2u$ is  bounded on $\mathbb R^{n+1}_-$, then  Theorem \ref{thm-Mainthm-Asym-InSummary} implies  the desired results. More explicitly, for any $(x_0,t_0)\in\mathbb R^{n+1}_-$ with $|x_0|$ and $R$ sufficiently large, we consider 
    \[
    u_R(x,t):=\frac{1}{R^2}u(R(x+x_0),R^2(t+t_0)),\quad\forall~(x,t)\in B_{6\sqrt n+1}\times\left[-\frac{1}{n},0\right].
    \]
    Notice that for all $(x,t)\in B_{6\sqrt n+1}\times\left[-\frac{1}{n},0\right]$, 
    \[
        R(x+x_0)\in B_{R(6\sqrt{n}+1)}(Rx_0),\quad R^2(t+t_0)\in \left[-\frac{R^2}{n}+t_0R^2,t_0R^2\right].
    \]
    Since $\mathrm{supp}(f)\subset B_S\times[-T,0]$, it follows that for all $|x_0|>6\sqrt{n}+1$, we may pick sufficiently large $R_1$ such that for all $R\geq R_1$, $(R(x+x_0),R^2(t+t_0))\in\mathbb R^{n+1}_-\setminus \mathrm{supp}(f)$. Especially, by triangle inequality,
    \[
    \sup_{x\in B_{6\sqrt n+1}}\left|u_R\left(x,-\frac{1}{n}\right)\right|
    \leq \frac{1}{R^2}\sup_{z\in B_{(6\sqrt n+1)(R+|x_0|)}}\left|
        u\left(z,-\frac{R^2}{n}+t_0R^2\right)
    \right|.
    \]
    Using the growth condition \eqref{equ-cond-Growth-BhattaWarrenWeser}, for any $\epsilon_0<1$ sufficiently small, we may choose $R_2>R_1$ such that for all $R\geq R_2$, 
    \[
    \begin{array}{llllll}
        \displaystyle \sup_{z\in B_{(6\sqrt n+1)(R+|x_0|)}}\left|
            u\left(z,-\frac{R^2}{n}+t_0R^2\right)
        \right| &\leq & \displaystyle \frac{1}{(6\sqrt n+2-\epsilon_0)^2}\left((6\sqrt n+1)^2(R+|x_0|)^2+R_0\right)\\
        &\leq & \displaystyle \frac{R^2}{1+\epsilon_1}\left(
            \left(1+\frac{|x_0|}{R}\right)^2+\frac{R_0}{R^2}
        \right),
    \end{array}
    \]
    for some $\epsilon_1>0$ independent of $R$. Therefore, choosing $R_2$ even lager, it proves 
    \[
    \sup_{x\in B_{6\sqrt n+1}}\left|u_R\left(x,-\frac{1}{n}\right)\right|\leq 1,\quad\forall~R\geq R_2.
    \]
    Applying Proposition 4.2 in \cite{Bhattacharya-Warren-Weser-Liouville-LagranFlow-CPDE}, since $u$ is a convex, classical solution to \eqref{equ-LagFlow-Extended}, it proves 
    \[
    D^2u(Rx_0,R^2t_0)=D^2u_R(0,0)\leq C(n),\quad\forall~|x_0|>\max\{6\sqrt n+1,S\},~t_0\in(-\infty,0].
    \]
    Therefore, together with the assumption that $u\in C^{2,1}(\mathbb R^{n+1}_- )$, we conclude that $D^2u$ is bounded in $\mathbb R^{n+1}_-.$ This finishes the proof of Theorem \ref{thm-Main1-AsymBehav}.
\end{proof}

At the end, we provide the optimality of the convergence rate \eqref{equ-Result-AsymBehav-u} in the following sense. 
\begin{proof}[Proof of Remark \ref{Rem-1.3}]
    Since $u$ is strictly convex with respect to $x$,  $F(D^2u)=\sum_{i=1}^n\arctan\lambda_i(D^2u)$ becomes a concave operator \cite{Caffarelli-Nirenberg-Spruck-DirichletIII}. 
    As in Section \ref{seclabel-Sec-ExistenceSol}, we may assume without loss of generality that $A=\mathrm{diag}(a_1,a_2,\cdots,a_n)$ is a diagonal matrix.
    Therefore,  $u$ satisfies 
    \[
    \begin{array}{llll}
    u_t &=&\displaystyle \sum_{i=1}^n\arctan\lambda_i(D^2u)+f_\alpha(x,t)\\
    &
    \leq&\displaystyle \sum_{i}^n\arctan\lambda_i(A)+
    \sum_{i=1}^n\frac{1}{1+a_i^2}(D_{ii}u-a_i)+f_\alpha(x,t),\quad\forall~(x,t)\in\mathbb R^n\times(-T,0],
    \end{array}
    \]
    and $E$ satisfies 
    \[
    E_t=u_t-\tau\leq \sum_{i=1}^n\frac{1}{1+a_i^2}D_{ii}E(x,t)+f_\alpha(x,t),\quad\forall~(x,t)\in\mathbb R^n\times(-T,0].
    \] 
    Since $E(x,-T)= 0$ in $\mathbb R^n$, $E(x,t)=o(1)$ as $-t+|x|^2\rightarrow\infty$, it follows from the fundamental solution (see for instance \cite{Book-Friedman-PDE-ParaboForm}) and the comparison principle that 
    \[
    E(x,t)\leq \int_{-T}^t\int_{\mathbb R^n}\frac{1}{(4\pi)^{\frac{n}{2}}}\frac{1}{(t-\tau)^{\frac{n}{2}}}\exp\left(-\sum_{i=1}^n\frac{(1+a_i^2)|x_i-y_i|^2}{4(t-\tau)}\right)f_\alpha(y,\tau)\mathrm dy\mathrm d\tau<0,
    \]
    for all $(x,t)\in\mathbb R^n\times(-T,0].$ 
    For any $y\in\overline{B_S}$,
    \[
    |x_i-y_i|^2=|x_i^2|-2x_iy_i+|y_i|^2
    \leq |x_i|^2+2S|x_i|+S^2.
    \]
    Consequently, using the assumption that $f_\alpha(x,t)\leq -(t+T)^\alpha$ in $B_{\frac{S}{2}}\times[-T,0]$, it follows that
    \[
    \begin{array}{lll}
     E(x,t)& \leq & \displaystyle \int_{-T}^t\int_{B_{\frac{S}{2}}}\frac{1}{(4\pi)^{\frac{n}{2}}}\frac{1}{(t-\tau)^{\frac{n}{2}}}\exp\left(-\sum_{i=1}^n\frac{(1+a_i^2)|x_i-y_i|^2}{4(t-\tau)}\right)f_\alpha(y,\tau)\mathrm dy\mathrm d\tau\\
     &\leq &\displaystyle  -\frac{1}{(4\pi)^{\frac{n}{2}}}
     \int_{0}^{t+T}\int_{B_{\frac{S}{2}}}
     \dfrac{1}{s^{\frac{n}{2}}}\exp\left(
        -\sum_{i=1}^n\frac{(1+a_i^2)\left(|x_i|^2+2S|x_i|+S^2\right)}{4s}
     \right)(t+T-s)^\alpha\mathrm dy \mathrm ds\\
     &\leq &\displaystyle  -C\int_0^{t+T}\frac{1}{s^{\frac{n}{2}}}\exp\left(
        -\sum_{i=1}^n\frac{(1+a_i^2)\left(|x_i|^2+2S|x_i|+S^2\right)}{4s}
     \right)(t+T-s)^\alpha\mathrm ds,
    \end{array}
    \]
    where $C>0$ is a positive constant relying only on $n$ and $S$. For sufficiently large $|x|$, we have a positive constant $C>1$ such that 
    \[
    C^{-1}x'(I+A^2)x\leq \mathfrak D:=\sum_{i=1}^n\frac{(1+a_i^2)\left(|x_i|^2+2S|x_i|+S^2\right)}{4}\leq Cx'(I+A^2)x.
    \]
    For any $0<\beta<1$ and $n\geq 4$ taking integration by part once,  we have 
    \[
    \begin{array}{lllll}
        |E(x,t)| & \geq &\displaystyle C\int_0^{(1-\beta)(t+T)}\frac{1}{s^{\frac{n}{2}}}\exp\left(
        -\frac{\mathfrak D}{s}
     \right)(t+T-s)^\alpha\mathrm ds\\
     &\geq &\displaystyle C\beta^\alpha(t+T)^\alpha\int_0^{(1-\beta)(t+T)}\frac{1}{s^{\frac{n}{2}}}\exp\left(
        -\frac{\mathfrak D}{s}
     \right)\mathrm ds\\
     &\geq &\displaystyle C\frac{\beta^\alpha}{(1-\beta)^{\frac{n}{2}-2}}\frac{1}{(t+T)^{\frac{n}{2}-2-\alpha}\mathfrak D}\exp\left(-\frac{\mathfrak D}{(1-\beta)(t+T)}\right)\\
     &&\displaystyle +C\beta^\alpha  \left(\frac{n}{2}-2\right)\frac{(t+T)^\alpha}{\mathfrak D}\int_0^{(1-\beta)(t+T)}
    s^{1-\frac{n}{2}}\exp\left(-\frac{\mathfrak D}{s}\right)\mathrm d s\\
     &\geq & \displaystyle C\frac{1}{(t+T)^{\frac{n}{2}-2-\alpha}x'(I+A^2)x}\exp\left(
        -\frac{\mathfrak D}{(1-\beta)(t+T)}
     \right),
    \end{array}
    \]
    where $C$ denotes positive constants rely on $n,S,\alpha$ and $\beta$. 
    When $n=2,3$, we take integration by part twice to obtain
    \[
    \begin{array}{lllll}
        |E(x,t)| &\geq & \displaystyle C \frac{1}{(t+T)^{\frac{n}{2}-2-\alpha}\mathfrak D}\exp\left(-\frac{\mathfrak D}{(1-\beta)(t+T)}\right)\\
        &&\displaystyle -C \frac{1}{(t+T)^{\frac{n}{2}-1-\alpha}\mathfrak D^2}\exp\left(-\frac{\mathfrak D}{(1-\beta)(t+T)}\right)\\
        &&\displaystyle +C\left(\frac{n}{2}-2\right)\left(\frac{n}{2}-3\right)\frac{(t+T)^\alpha}{\mathfrak D^2}\int_0^{(1-\beta)(t+T)}s^{2-\frac{n}{2}}
        \exp\left(-\frac{\mathfrak D}{s}\right)\mathrm ds\\
        &\geq & \displaystyle C \frac{1}{(t+T)^{\frac{n}{2}-2-\alpha}\mathfrak D}\exp\left(-\frac{\mathfrak D}{(1-\beta)(t+T)}\right)
    \end{array}
    \]
    where $C$ denotes positive constants rely on $n,S,\alpha,\beta$, and $|x|$ is taken sufficiently large to make $\mathfrak D$ large enough.
    In summary, the estimates above imply 
    \[
    \begin{array}{llll}
    |E(x,t)| & \geq & \displaystyle \frac{C}{(t+T)^{\frac{n}{2}-\alpha-2}x'(I+A^2)x}\exp\left(-\frac{\mathfrak D}{(1-\beta)(t+T)}\right)\\
    &= & \displaystyle  \frac{C}{(t+T)^{\frac{n}{2}-\alpha-2}x'(I+A^2)x}\exp\left(-\sum_{i=1}^n\frac{(1+a_i^2)\left(|x_i|^2+2S|x_i|+S^2\right)}{4(1-\beta)(t+T)}\right)\\
    &\geq & \displaystyle \frac{C}{(t+T)^{\frac{n}{2}-\alpha-2}x'(I+A^2)x}\exp\left(- \frac{x'(I+A^2)x+4S|(I+A^2)x|}{4(1-\beta)(t+T)}\right),
    \end{array}
    \]
    for sufficiently large $|x|$ and for all $t\in (-T,0]$. 
    Consequently, the main term in estimate \eqref{equ-Result-AsymBehav-u} cannot be enhanced in general.
\end{proof}

\small

\bibliographystyle{plain}

\bigskip

\noindent J. Bao 

\medskip

\noindent  School of Mathematical Sciences, Beijing Normal University\\
Beijing 100875, China \\[1mm]
Email: \textsf{jgbao@bnu.edu.cn}

\bigskip

\noindent Z. Liu (Corresponding author)

\medskip

\noindent  School of Mathematics, Statistics and Mechanics, Beijing University of Technology\\
Beijing 100124, China \\[1mm]
Email: \textsf{liuzixiao@bjut.edu.cn}

\end{document}